\documentclass[a4paper,12pt]{article} 
\usepackage{amsfonts, amsthm, amsmath, amssymb}
\usepackage[usenames]{color}
\usepackage{graphicx}
\usepackage[all]{xy}
\usepackage[top=3cm,bottom=3cm,left=3cm,right=3cm]{geometry}
\usepackage[pdftex]{hyperref}

\newcommand{\pd}{\partial} 
\newcommand{\sbxt}{{\scriptstyle\boxtimes}}

\newcommand{\bC}{\mathbb{C}} 
\newcommand{\bM}{\mathbb{M}} 
\newcommand{\bN}{\mathbb{N}} 
\newcommand{\bP}{\mathbb{P}}
\newcommand{\bQ}{\mathbb{Q}}
\newcommand{\bR}{\mathbb{R}}
\newcommand{\bZ}{\mathbb{Z}}
\newcommand{\unit}{\mathbf{1}} 

\newcommand{\cC}{\mathcal{C}}
\newcommand{\cH}{\mathcal{H}}
\newcommand{\cO}{\mathcal{O}}
\newcommand{\cU}{\mathcal{U}}
\newcommand{\cW}{\mathcal{W}}
\newcommand{\cY}{\mathcal{Y}}

\newcommand{\fH}{\mathfrak{H}}

\newcommand{\fX}{\mathfrak{X}}

\newcommand{\Ann}{\operatorname{Ann}}
\newcommand{\Aut}{\operatorname{Aut}}
\newcommand{\End}{\operatorname{End}}
\newcommand{\Hom}{\operatorname{Hom}}

\newcommand{\Irr}{\operatorname{Irr}}
\newcommand{\Ker}{\operatorname{Ker}}

\newcommand{\Tr}{\operatorname{Tr}}
\newcommand{\wt}{\operatorname{wt}}

\newtheorem{thm}{Theorem}[section]
\newtheorem{prop}[thm]{Proposition}
\newtheorem{cor}[thm]{Corollary}
\newtheorem{lem}[thm]{Lemma}

\newtheorem*{uthm}{Theorem}
\newtheorem*{conj}{Conjecture}

\theoremstyle{definition}
\newtheorem{defn}[thm]{Definition}

\theoremstyle{remark}
\newtheorem{rem}[thm]{Remark}

\numberwithin{equation}{section}

\begin{document}
\title{Regularity of fixed-point vertex operator subalgebras}
\author{\begin{tabular}{c}
Scott Carnahan\footnote{Supported by the Program to Disseminate Tenure Tracking System, MEXT, Japan} \\
and \\
Masahiko Miyamoto\footnote{Supported
by a Grant-in-Aid for Scientific Research, No. 26610002, The Ministry of Education,
Science and Culture, Japan}\vspace{-0mm}\\
Institute of Mathematics, University of Tsukuba, \vspace{-2mm}\\
Tsukuba 305 Japan \end{tabular}}
\date{}
\maketitle

\begin{abstract}
We show that if $T$ is a simple non-negatively graded regular vertex operator algebra with a nonsingular invariant bilinear form and $\sigma$ is a finite order automorphism of $T$, then the fixed-point vertex operator subalgebra $T^\sigma$ is also regular.  This yields regularity for fixed point vertex operator subalgebras under the action of any finite solvable group.  As an application, we obtain an $SL_2(\bZ)$-compatibility between twisted twining characters for commuting finite order automorphisms of holomorphic vertex operator algebras.  This resolves one of the principal claims in the Generalized Moonshine conjecture.
\end{abstract}

\tableofcontents

\section{Introduction}

Let $T$ be a vertex operator algebra (abbreviated as VOA).  If $T$ satisfies some property $P$, it is natural to ask if, for any finite order automorphism $\sigma$, the fixed-point vertex operator subalgebra $T^\sigma$ also satisfies $P$.  We call this question the cyclic orbifold problem for $P$, since the chiral algebras of orbifold CFT models are given by fixed-point vertex operator subalgebras.  In this paper, we are mostly concerned with a property called ``regularity'' that combines finiteness with complete reducibility.

For our main theorem, we show that if $T$ is simple, regular, non-negatively graded, and equipped with a nonsingular invariant bilinear form, then so is $T^\sigma$.  That is, we resolve the cyclic orbifold problem for the ``simple, regular, non-negatively graded, with nonsingular invariant form'' property.  While this property appears to be rather long and restrictive, many fundamentally important VOAs, such as unitary minimal models, rational WZW models, lattice VOAs, the Monster VOA $V^\natural$ appearing in moonshine, and finite tensor products of them have this property.  Furthermore, there are many substantial results, such as Zhu's theorem on modular invariance of trace functions \cite{Zh} and the Verlinde Conjecture \cite{HV}, that are consequences of this property.  However, regularity can be quite difficult to prove for a particular VOA without substantial control of its internal structure.  Thus, a resolution of the cyclic orbifold problem allows us to apply powerful tools where they were once useless.

The cyclic orbifold problems for the non-negatively graded property and existence of a nonsingular invariant form are quite trivial, and the cyclic orbifold problem for simplicity is resolved by Theorem 3 of \cite{DM}.  It remains to show that regularity is preserved under our conditions, that is, if any weak $T$-module is a direct sum of irreducible ordinary $T$-modules, then any weak $T^\sigma$-module is a direct sum of irreducible ordinary $T^\sigma$-modules.

We shall reduce this question further: under the non-negative grading hypothesis, regularity is equivalent to the combination of a finiteness property and a weaker complete reducibility assertion.  The finiteness property is a condition we call $C_2^0$-cofiniteness, and it is a minor variant of the $C_2$-cofiniteness property introduced in \cite{Zh} as a technical condition to prove a modular invariance property of the space spanned by trace functions on modules.  The $C_2$ condition is now understood as one of the most important finiteness conditions for a VOA, and we claim the same is true for $C_2^0$-cofiniteness.  Recently, the cyclic orbifold problem for the $C_2^0$ cofiniteness property was resolved \cite{M3}, so our primary goal is to prove the other condition, namely that irreducible $T^\sigma$-modules are projective in the category of finitely generated logarithmic $T^\sigma$-modules.

We prove the projectivity of irreducible modules by appealing to their rigidity, and the proof of rigidity is the most intricate part of the paper, drawing on substantial manipulation of genus one data.

One may ask if we can weaken our hypotheses further.  It is conceivable that one could weaken the non-negatively graded condition, but it would require even more examination of the literature than we have done in section \ref{sec:literature}.  We use the grading assumption for $T^\sigma$ in the modular invariance result of \cite{M04}, and by the remark the end of the introduction in that paper, it appears that we may need to assume the stronger condition of $C_{2+s}$-cofiniteness when $T^\sigma$ has lowest $L(0)$-eigenvalue $-s \leq 0$.  We also use the assumption $s_{T,T}\neq 0$, which follows from the Verlinde formula, whose proof in \cite{HV} assumes $T$ is of ``CFT type''.  We have been able to replace this assumption with the strictly weaker ``non-negatively graded'' assumption, but going further appears to require substantially more work.  As far as we know, it is not so easy to find interesting examples that would justify the effort.

At the end of this paper, we will describe an application of our main theorem to the theory of holomorphic orbifolds and the Generalized Moonshine conjecture.  Further applications have already appeared in the literature.  For example, \cite{vEMS} use our main result to construct new abelian intertwining algebras, and place the theory of cyclic orbifolds on a very sound footing.  Their results have allowed the classification of holomorphic vertex operator algebras of central charge 24 to advance quite rapidly, and have also led to the full solution of the Generalized Moonshine conjecture in \cite{GM4}.

This paper has the following structure: In \S2, we introduce some definitions and notation.  In \S3, we prove some auxiliary results that connect our methods with the existing literature.  In \S4, we prove three useful results: First, $T$ decomposes as a direct sum of simple currents for $T^\sigma$.  Second, any irreducible $T^\sigma$-module is a direct summand of some irreducible twisted $T$-module.  Third, $T^{\sigma}$ is projective as a $T^{\sigma}$-module.  This section is ``purely genus zero'' in the sense that the contribution of modular invariance results is negligible.  In \S5, we introduce the Moore-Seiberg-Huang argument using the machinery of two-point genus-one functions, prove the rigidity of every simple $T^{\sigma}$-module, and conclude with the Main Theorem.  In \S6, we apply our main theorem to produce an $SL_2(\bZ)$-compatibility for twisted twining characters in the theory of holomorphic orbifolds, and in particular, resolve the corresponding $SL_2(\bZ)$-compatibility claim in the Generalized Moonshine Conjecture.

\subsection{Main results}

\begin{uthm} (Corollary \ref{cor:irreducible-induced-modules})
Let $T$ be a simple non-negatively graded regular vertex operator algebra with a nonsingular invariant bilinear form and let $\sigma$ be a finite order automorphism of $T$.  Then any irreducible $T^\sigma$-module is a direct summand of some irreducible twisted $T$-module, where the twisting is by some power of $\sigma$.
\end{uthm}

The following is our main theorem in its strong form.  We initially prove the case of finite cyclic groups, but the case of finite solvable groups follows immediately.

\begin{uthm} (Corollary \ref{cor:solvable-fixed-points})
Let $T$ be a simple non-negatively graded regular vertex operator algebra with a nonsingular invariant bilinear form and let $G$ be a finite solvable group of automorphisms of $T$.  Then the fixed-point vertex operator subalgebra $T^G$ is simple, non-negatively graded, regular, and admits a non-singular invariant bilinear form (i.e., it is self-dual as a $T^G$-module).
\end{uthm}

The following is one of the main assertions in Norton's Generalized Moonshine conjecture, once we set $V$ to be the Moonshine Module $V^\natural$.

\begin{uthm} (Theorem \ref{thm:sl2z-rule})
Let $V$ be a holomorphic $C_2$-cofinite vertex operator algebra with non-negative $L(0)$-spectrum.  Given a commuting pair $(g,h)$ of finite order automorphisms of $V$, let $Z(g,h;\tau) = \Tr(\tilde{h} q^{L_0-1}|V(g))$, where $V(g)$ is an irreducible $g$-twisted $V$-module, and $\tilde{h}$ is some lift of $h$ to a linear transformation on $V(g)$.  Then for any $\left(\begin{smallmatrix} a & b \\ c & d \end{smallmatrix} \right) \in SL_2(\bZ)$, the holomorphic function $\tau \mapsto Z(g,h,\frac{a\tau+b}{c\tau+d})$ is proportional to the function $\tau \mapsto Z(g^a h^c, g^b h^d,\tau)$.
\end{uthm}

\noindent
{\bf Acknowledgement} \\
S. C. would like to thank J. van Ekeren, Y.-Z. Huang, K. Kawasetsu, C.-H. Lam, and J. Lepowsky for helpful comments and corrections to earlier versions of this article.
M. M. would like to express special thanks to H.~Yamauchi and T.~Arakawa for giving 
him the chance to explain the details of his proofs.  
He also gives thanks to T.~Abe and A.~Matsuo for their advice. 

\section{Notation conventions}

A VOA of central charge $c \in \bC$ \cite{FLM} is a quadruple $(V, Y(-,z), \unit, \omega)$, where $V$ is a vector space with distinguished vectors $\unit, \omega \in V$, and $Y: V \otimes V \to V((z))$ is a multiplication map, with multiplication written as $(u,v) \mapsto Y(u,z)v = \sum_{n \in \bZ} u_n v z^{-n-1}$.  These data are required to satisfy the following axioms:
\begin{enumerate}
\item (unit) $Y(\unit,z)v = \unit_{-1} v = v$ and $Y(v,z)\unit \in v + zV[[z]]$ for all $v \in V$.
\item (Virasoro) The coefficients of the expansion $Y(\omega,z) = \sum_{n \in \bZ} L(n) z^{-n-2}$ induce a central charge $c$ action of the Virasoro algebra, i.e., $[L(m), L(n)] = (m-n)L(m+n) + \frac{m^3-m}{12} c \delta_{m,0} Id_V$.
\item (derivation) $\frac{d}{dz} Y(v,z) = Y(L(-1)v,z)$ for all $v \in V$.
\item (grading) The operator $L(0)$ acts semisimply with integer eigenvalues (called weights) that are bounded below and have finite multiplicity.  This endows $V$ with a grading $V = \bigoplus_{n \in \bZ} V_n$, where $V_n = \{ v \in V | L(0)v = nv\}$ is finite dimensional.  When $v \in V_n$, we write $\wt(v) = n$.
\item (Jacobi identity) If we let $z^{-1} \delta\left(\frac{x-y}{z}\right) = \sum_{m \geq 0, n \in \bZ} (-1)^m \binom{n}{m} x^{n-m} y^m z^{-n-1}$, then
\[ \begin{aligned}
z^{-1} \delta&\left(\frac{x-y}{z}\right) Y(u, x) Y(v,y) - z^{-1} \delta\left(\frac{y-x}{-z}\right) Y(v,y) Y(u, x) \\
&= y^{-1} \delta\left(\frac{x-z}{y}\right) Y(Y(u, z)v, y) 
\end{aligned} \]
\end{enumerate}
In this paper, we will only consider VOAs that are non-negatively graded, i.e., with non-negative $L(0)$-eigenvalues.

For the representation theory of $V$, we are primarily concerned with a class of modules that satisfy some nice finiteness properties without being overly restricted.  The most general class we consider is a weak $V$-module, i.e., a vector space $U$ with an action map $Y^U: V \otimes U \to U((z))$, satisfying:
\begin{enumerate}
\item (unit) $Y^U(\unit,z)u = \unit_{-1} u$ for all $u \in U$
\item (Virasoro) The coefficients of the expansion $Y^U(\omega,z) = \sum_{n \in \bZ} L(n) z^{-n-2}$ induce a central charge $c$ action of the Virasoro algebra, where $c$ is the central charge of $V$.
\item (Jacobi identity) This is the same formula as for VOAs, but with $Y^U$ in place of $Y$ (where applicable).
\item (derivation) $\frac{d}{dz} Y^U(v,z) = Y^U(L(-1)v,z)$ 
\end{enumerate}
We will use the term ``$V$-module'' to mean a finitely generated weak $V$-module that admits a decomposition into generalized eigenspaces for $L(0)$.  This would be called a ``finitely generated logarithmic $V$-module'' following \cite{M02} and a ``finitely generated generalized $V$-module'' following \cite{HLZ1}.  An admissible $V$-module is a weak $V$-module that admits a $\bZ_{\geq 0}$-grading that is compatible with the $L(0)$-grading on $V$.  Note that all $V$-modules are admissible $V$-modules, but not all admissible $V$-modules are $V$-modules in our sense.  An ordinary $V$-module is a weak $V$-module on which $L(0)$-acts semisimply, with eigenvalues that are bounded below in each coset of $\bZ$ and eigenvalue multiplicities that are finite.  Again, ordinary $V$-modules are not necessarily $V$-modules in our sense, since one may consider an infinite direct sum of modules with differing lowest conformal weight.  For example, a positive definite lattice VOA is an ordinary module for the underlying Heisenberg VOA, but not a module.

We will work with a minor variant of Zhu's $C_2$-cofiniteness condition.  For any $V$-module $W$ and any $n \geq 1$, we define $C_n^0(W)$ to be the subspace of $W$ spanned by $\{v_{-n}u\mid v\in V, w \in W, \wt(v)\geq 1\}$.  The usual $C_n(W)$ lacks the condition on the weight of $v$, so we clearly have $C_n^0(W) \subseteq C_n(W)$.  When we say that $V$ is $C_2^0$-cofinite, we mean that $\dim V/C_2^0(V)<\infty$.  When $V$ is ``CFT type'', i.e., non-negatively graded and $\Ker L(0)$ is spanned by $\unit$, the two conditions become equivalent, but otherwise our condition is \textit{a priori} slightly stronger.  That is, the following implications are straightforward:

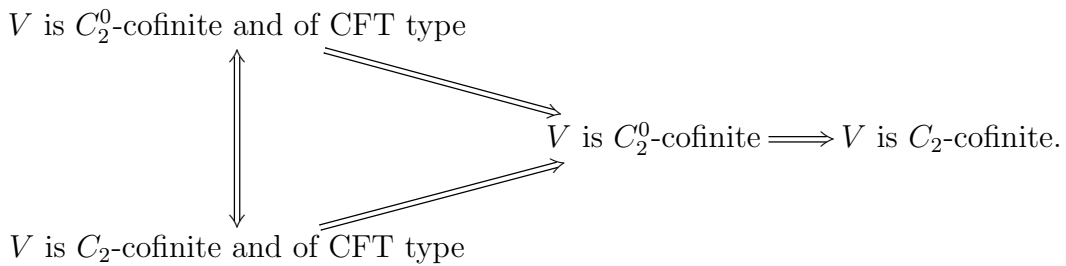
\[ \xymatrix{ V \text{ is } C_2^0 \text{-cofinite and of CFT type} \ar@{<=>}[dd] \ar@{=>}[dr] \\ & V \text{ is } C_2^0 \text{-cofinite} \ar@{=>}[r] & V \text{ is } C_2 \text{-cofinite.} \\ V \text{ is } C_2 \text{-cofinite and of CFT type} \ar@{=>}[ur] } \]

The main advantage of $C_2^0$-cofiniteness over $C_2$-cofiniteness is that the cyclic orbifold problem is solved for $C_2^0$-cofiniteness \cite{M3} without the CFT type condition.

For any (finitely generated logarithmic) $V$-modules $A,B,C$, we will use the term ``intertwining operator of type $\binom{C}{A,B}$'' to denote a logarithmic intertwining operator as described in Definition 3.10 of \cite{HLZ2}: it is a map $A \otimes B \to C[\log z]\{z\}$ satisfying a lower truncation condition, the $L(-1)$-derivative property $\cY(L(-1)a,z)=\frac{d}{dz}\cY(a,z)$, and the Jacobi identity.  As a consequence of our finiteness conditions (see e.g., Lemma \ref{lem:tensor-products-work-well}), the series $\cY(a,z)$ has the form $\cY(a,z)=\sum_{j=0}^K\sum_{m\in \bC}a_{j,m}z^{-m-1}\log^j z$, i.e., $\cY$ can be viewed as a map $A \otimes B \to C\{z\}[\log z]$.  We will write $I\binom{C}{A,B}$ for the vector space of all intertwining operators of type $\binom{C}{A,B}$, including those with logarithmic terms.


If $V$ is non-negatively graded and $C_2^0$-cofinite, then we will see in Lemma \ref{lem:tensor-products-work-well} that the category of $V$-modules satisfies the conditions given in \cite{HP} that are sufficient for the Huang-Lepowsky-Zhang tensor product theory to function.  Given $V$-modules $A$ and $B$, there is a pair $(A\boxtimes B, \cY^{\boxtimes}_{A,B})$ given by a (finitely generated logarithmic) $V$-module $A\boxtimes B$ and a surjective (logarithmic) intertwining operator $\cY_{A,B}^{\boxtimes}$ of type $\binom{A \boxtimes B}{A,B}$ such that for any $\cY\in I\binom{C}{A,B}$, there is a homomorphism $\phi:A\boxtimes B \to C$ such that $\cY=\phi\cY_{A,B}^{\boxtimes}$.  Although the pair $(A\boxtimes B, \cY_{A,B}^{\boxtimes})$ is not uniquely determined by this property, it is uniquely determined up to isomorphism.  In order to get a true universal property, one needs an additional compatibility condition, and we will use the $P(1)$ compatibility from Huang-Lepowsky-Zhang, i.e., we set $A \boxtimes B = A \boxtimes_{P(1)} B$.  By the isomorphism given in \cite{A} between $V$-module intertwining operators and genus zero three point conformal blocks, it is equivalent to say that the pair $(A\boxtimes B, \cY_{A,B}^{\boxtimes})$ represents the functor that takes a $V$-module $U$ the the space of $V$-module conformal blocks on $\bP^1$ with marked points $0,1,\infty$ with the standard coordinates, and insertions of $A$ at $0$, $B$ at $1$, and $U'$ at $\infty$.

We fix a representative $(A\boxtimes B, \cY_{A,B}^{\boxtimes})$ of the isomorphism class for each pair $(A,B)$, by setting $A \boxtimes B = A \boxtimes_{P(1)} B$ in the notation of Huang-Lepowsky-Zhang.


Given $V$-module homomorphisms $\alpha: A\to C$ and $\beta: B\to D$, the functoriality of $P(1)$-tensor products yields an induced homomorphism from $A\boxtimes B$ to $C\boxtimes D$, which 
we denote by $\alpha \sbxt \beta$.  We write $\sigma\circ\delta$ to denote the composition of homomorphisms $\sigma$ and $\delta$. 

We follow the convention used in Huang's papers (e.g., \cite{HD}) to describe equality of $n$-point functions: When we say that an identity (of formal power series with fractional powers and logarithms) holds on a complex domain $\cO$, we implicitly mean that one chooses a set of branch cuts on $\cO$ that yield a distinguished choice of fractional powers and logarithms of the corresponding variables, and then applies the substitution at each point to get a pair of series that converge to an identity of complex numbers.  Furthermore, when we claim equality, we implicitly claim that the two series converge absolutely uniformly on compact subsets of the cut region, so that we have an equality of holomorphic functions.  To ensure consistency with results in the literature, we choose our branch cuts according to the method given by Huang: arguments of variables are restricted to lie between zero and $2\pi$, and fractional powers and logarithms are then chosen accordingly.

\section{Literature adjustment} \label{sec:literature}

This section contains some minor lemmata that allow us to strengthen our main result and render it more palatable.

\subsection{Removing ``CFT type'' assumptions}

Some of the results that we use in this paper were initially proved under the assumption that $L(0)$ has non-negative spectrum and $\Ker L(0)$ is spanned by the vacuum vector.  This property is known as ``CFT type'' in some papers and ``positive energy'' in others.  In order to eliminate the assumption on $\Ker L(0)$ from our main theorem, we need show that the assumption can be removed from some results in the literature.  The situation is the following:
\begin{enumerate}
\item In \cite{HP}, some sufficient conditions are given for a VOA $V$ to have the property that all finitely generated $V$-modules have finite length, and one of those conditions is that $V$ be $C_2$-cofinite and positive energy.  We will need to use the finite-length property to reduce the regularity question to the condition that all irreducible finitely generated $V$-modules are projective.
\item In \cite{HV}, the main results are stated under the positive energy hypothesis.  We will need to use the result that if $T$ is regular, then the $S$-matrix element $s_{TT} \neq 0$.  The positive energy assumption is not directly used in Huang's paper, but some results in \cite{HD} that use this assumption are cited.  In particular, in \cite{HD} section 3, the author assumes all $V$-modules are $C_2$-cofinite and $\bR$-graded.  In the proof of Theorem 7.2, Huang cites \cite{ABD} Proposition 5.2 to deduce all hypotheses used earlier in the paper from the combination of regularity and the positive energy assumption, but notes that one only needs $C_2$-cofiniteness of modules, complete reducibility of $\bN$-gradable weak $V$-modules, and the non-negativity of the $L(0)$-spectrum on $V$. 
\item In \cite{M3}, the statement of the main theorem, concerning the resolution of the cyclic orbifold problem for $C_2$-cofiniteness, contains the assumption that $V$ is non-negatively graded and $V_0$ is spanned by the vacuum vector.  However, the mathematical content of the proof is a solution to the cyclic orbifold problem for $C_2^0$-cofiniteness without that assumption.
\end{enumerate}

Naturally, the reader who is only concerned with vertex operator algebras of CFT type can safely ignore this section.

\begin{lem} \label{lem:finite-length}
If $V$ is a $C_2^0$-cofinite VOA with non-negative $L(0)$-spectrum, then any $V$-module has finite length.  In particular, any finitely generated $V$-module contains an irreducible submodule.
\end{lem}
\begin{proof}
In \cite{HP}, the combination of Proposition 3.8 and Corollary 3.16 yields the assertion that finite length follows from the following assumptions:
\begin{enumerate}
\item $V$ is $C_1$-cofinite in the sense of Li, i.e., the subspace $C_1^a(V)$ spanned by $\{ u_{-1}v | \wt(u) > 0, \wt(v) > 0 \}$ and $\{ L(-1)v | v \in V \}$ satisfies $\dim V/C_1^a(V) < \infty$.  We will call this condition $C_1^a$-cofiniteness.
\item $A_0(V)$ is finite dimensional.
\end{enumerate}
Huang shows in Proposition 4.1 that these conditions are satisfied when $V$ is $C_2$-cofinite and positive-energy, but we wish to weaken the positive energy condition to the condition that the $L_0$-spectrum is non-negative.

For the first condition, we show that $C_2^0(V) \subseteq C_1^a(V)$.  By the derivative axiom, $(L(-1)u)_{-1} = u_{-2}$ for all $u \in V$.  Thus, $C_1^a(V)$ contains any vector of the form $u_{-2}v$, where $\wt(u) > 0$ and $\wt(v) > 0$.  By the non-negative grading assumption, it remains to consider vectors of the form $u_{-2}v$ where $\wt(u) > 0$ and $\wt(v) = 0$.  Applying skew-symmetry, we find that such a vector is equal to $L(-1)(v_{-1}u) - (L(-1)v)_{-1}u$, and each summand lies in $C_1^a(V)$.

For the second condition, finite-dimensionality of $A_N(V)$ for all $N \geq 0$ follows from non-negative grading and $C_2$-cofiniteness by Theorem 2.5 of \cite{M04}.
\end{proof}

\begin{lem} \label{lem:modules-are-c2}
Let $V$ be a $C_2$-cofinite VOA that is non-negatively graded.  Then if $W$ is any finitely generated $V$-module, then $W/C_2(W)$ is finite dimensional, i.e., $W$ is $C_2$-cofinite.
\end{lem}
\begin{proof}
This proof is adapted from Proposition 5.2 of \cite{ABD}, which has the same conclusion for irreducible $V$-modules under the additional assumption that the kernel of $L(0)$ is spanned by $\unit$.

Let $w^1,\ldots,w^r$ be a generating set of $W$.  By \cite{M04} Lemma 2.4, $W$ is spanned by vectors of the form $v^1(i_1)\cdots v^k(i_k)w^j$, where $v^i \in A$ for some finite set $A \subset V$ of homogeneous elements such that $V$ is spanned by $A$ and $C_2(V)$, and $i_1 < \cdots < i_k$.  By \textit{loc. cit.} Theorem 2.7, $W$ is $\bN$-gradable, so $v^k(i_k)w^j = 0$ whenever $i_k$ is sufficiently large (depending on the maximal weight of $A$ and the weights of $w^j$).  We conclude that there is some $N > 0$ such that if $v^1(i_1)\cdots v^k(i_k)w$ is nonzero with weight greater than $N$, then $i_1 \leq -2$.  This implies $W$ is spanned by the combination of $C_2(W)$ together with the finite dimensional subspace $W_{\leq N}$ spanned by homogeneous vectors of weight at most $N$.
\end{proof}

\begin{lem} \label{lem:tensor-products-work-well}
If $V$ is non-negatively graded and $C_2^0$-cofinite, then the category of $V$-modules is a braided monoidal category under $\boxtimes = \boxtimes_{P(1)}$.  All modules have $L(0)$-spectrum supported on a finite union of cosets of $\bZ$ in $\bQ$, and all intertwining operators have nonzero coefficients attached to $z^\alpha (\log z)^\beta$ for $(\alpha,\beta) \in D \times \{0,1,\ldots,r\}$ for $D$ a finite union of cosets of $\bZ$.  For any $k$ composable intertwining operators $\cY_1, \cdots, \cY_k$, and elements $w'_{(0)}, w_{(1)},\ldots,w_{(k+1)}$ of suitable modules, the genus zero $k$-point function $\langle w'_{(0)}, \cY_1(w_{(1)}, x_1) \cdots \cY_k(w_{(k)}, x_k) w_{(k+1)} \rangle$ converges absolutely uniformly on compact subsets of the open domain
\[ \{ (x_1,\ldots,x_k) \in \bC^k | |x_1| > |x_2| > \cdots > |x_k| > 0 \} \]
to a holomorphic function that extends to a multivalued function that is a solution to a regular singular differential equation, and the branch locus is supported on the divisors $x_i = x_j$, $x_i = 0$, and $x_i = \infty$.
\end{lem}
\begin{proof}
We first point out where the claims can be found, without mentioning the necessary hypotheses.  The braided monoidal category assertion follows from Theorem 12.15 of \cite{HLZ8}.  The last claim, about about $k$-point functions, is given as \cite{HLZ7} Theorem 11.8.  The claim about $L(0)$-spectrum follows from \cite{M04} Corollary 5.10.  The claim about the form of intertwining operators follows from the regular singular property of one-point functions and what we know about the $L(0)$-spectrum.  

The Huang-Lepowsky-Zhang claims depend on assumptions that are scattered through the opus, but they are gathered in section 4 of \cite{HP}.  In particular, Huang proves that the assumptions follow from the following hypotheses:
\begin{enumerate}
\item $V$ is $C_1^a$-cofinite (see the proof of Lemma \ref{lem:finite-length}).
\item There is some $N > 0$ such that $A_N(V)$ is finite dimensional, and such that for any two irreducible modules, the real parts of their lowest $L(0)$-eigenvalues differs by at most $N$.
\item Every irreducible $V$-module $W$ is $\bR$-graded and $C_1^0$-cofinite, i.e., the subspace $C_1^0(W)$ satisfies $\dim W/C_1^0(W) < \infty$.
\end{enumerate}
The $C_1^a$-cofiniteness is shown in Lemma \ref{lem:finite-length}.  The second condition follows from the fact that there are only finitely many isomorphism classes of irreducible $V$ modules, together with the finite-dimensionality result of \cite{M04} Theorem 2.5.  The $\bR$-graded property is implied by the $L(0)$-spectrum claim, and the $C_1^0$-cofiniteness of modules follows from $C_2$-cofiniteness of modules given in Lemma \ref{lem:modules-are-c2}.  Specifically, the derivative axiom implies $u_{-2} w = (L(-1) u)_{-1} w$, meaning $C_2(W) \subseteq C_1^0(W)$.
\end{proof}

\begin{lem} \label{lem:stt-nonzero}
Let $T$ be a regular VOA that is non-negatively graded.  Then the $S$-matrix element $s_{TT} \neq 0$.
\end{lem}
\begin{proof}
Under the additional assumption that $\Ker L(0)$ is spanned by $\unit$, this is a consequence of Theorem 5.5 of \cite{HV}, so it suffices to show that this assumption is unnecessary.  As we mentioned earlier in this section, Huang's proof does not directly use this assumption, but it depends on results in \cite{HD} that are stated with this assumption.  In fact, the proofs there only use three weaker conditions that hold true for us: $C_2$-cofiniteness of irreducible $V$-modules is given by Lemma \ref{lem:modules-are-c2}, complete reducibility of $\bN$-gradable weak $V$-modules follows immediately from regularity, and the non-negativity of the $L(0)$-spectrum on $V$ is by assumption.
\end{proof}

\subsection{Regularity}

We need to relate the notion of regularity to the notions that are actually used in this paper.  First, we wish to show that regularity implies $C_2^0$-cofiniteness.  Li showed that regularity implies $C_2$-cofiniteness in \cite{Li99}, and his proof adapts well to our slightly stronger statement.

\begin{lem}
Let $V$ be a regular VOA.  Then for any $n \geq 2$, $V$ is $C_n^0$-cofinite.
\end{lem}
\begin{proof}
By Corollary 2.6 and Corollary 2.11 of \cite{Li99}, a subspace $U$ of a regular VOA $V$ satisfies $\dim V/U < \infty$ if and only if for any $v \in V$, there exists some $k \in \bZ$ such that for all $m \leq k$, $v_m V \subseteq U$.  We therefore apply this criterion to $U = C_n^0(V)$.

If $\wt(v) > 0$, then $k = -n$ is clearly satisfactory.  If $\wt(v) \leq 0$, then we let $j = 1 - \wt(v)$, so the $L(-1)$-derivative property implies $(-m-1)\cdots(-m-j) v_m u = (L(-1)^j v)_{m+j} u \in C_n^0(V)$ for all $u \in V$ and all $m \leq -j-n$.  Thus, we may take $k = \wt(v)-n-1$ when $\wt(v) \leq 0$.
\end{proof}

Our main theorem asserts that a VOA satisfying certain conditions is regular, but the Moore-Seiberg-Huang machinery yields a statement that may appear to be weaker, namely that any irreducible module is projective in the category of finitely generated (logarithmic) modules.  Regularity is already well-studied in the literature, and so we don't need to consider any new techniques to obtain our reduction.

\begin{lem} \label{lem:semisimple-anv}
Let $V$ be a non-negatively graded $C_2^0$-cofinite VOA, and suppose any surjection from a (finitely generated) $V$-module $U$ to a simple $V$-module $W$ splits.  Then the finite dimensional associative algebras $A_n(V)$ are semisimple for all $n \in \bZ_{\geq 0}$.
\end{lem}
\begin{proof}
This is essentially Theorem 4.10 of \cite{DLM2}, and we follow the same strategy, but our hypotheses are slightly different.  For any finite dimensional $A_n(V)$-module $M$, we consider the induced $V$-module $U = L_n(V_n \oplus M)$.  $U$ is finitely generated, because we may use any basis of the finite dimensional space $V_n \oplus M$.

We claim that $U$ splits as a direct sum of irreducible $V$-modules.  Let $soc(U)$ be the socle, i.e., the direct sum of irreducible $V$-submodules of $U$.  By Lemma \ref{lem:finite-length}, the quotient $V$-module $U/soc(U)$ has finite length, so if it is nonzero, then there is some irreducible $V$-submodule $W$.  By projectivity of $W$, the projection from the preimage of $W$ in $U$ to $W$ admits a splitting.  Thus, the preimage of $W$ lies in the socle, and we conclude that $U/soc(U)$ is zero.

Thus, $\Omega_n(U)/\Omega_{n-1}(U) \cong V_n \oplus M$ is a direct sum of finite dimensional irreducible $A_n(V)$-modules, and hence the same is true of $M$.  Complete reducibility for all finite dimensional modules of a finite dimensional unital associative algebra implies the algebra is semisimple.
\end{proof}

\begin{prop} \label{prop:regularity-from-projectivity}
Suppose $V$ is a non-negatively graded $C_2^0$-cofinite VOA, such that any irreducible $V$-module is projective in the category of (finitely generated logarithmic) $V$-modules.  Then $V$ is regular.
\end{prop}
\begin{proof}
Let $P$ be a weak $V$-module.  By Theorem 2.7 of \cite{M04}, $P$ decomposes as $\bigoplus_{n=0}^\infty P(n)$, where each $P(n)$ is a direct sum of generalized eigenspaces for $L(0)$.  In particular, $P$ is admissible in the sense of \cite{DLM3}.  

By Theorem 4.11 of \cite{DLM2}, if $A_n(V)$ is semisimple for all sufficiently large $n$, then $V$ is rational, i.e., any admissible $V$-module is a direct sum of irreducible admissible submodules.  Thus, we conclude rationality from Lemma \ref{lem:semisimple-anv}.  

Any irreducible admissible $V$-module is necessarily ordinary, since $L(0)$-acts semisimply.  We conclude that any weak module is a direct sum of irreducible ordinary modules, i.e., $V$ is regular.
\end{proof}

\subsection{Analytic lemmata}

We add some minor lemmata about intertwining operators that we could not find in the literature.  For convenience, we introduce two analytic open sets:
\[ \begin{aligned}
\cO_2 &=\{(x,y)\in \bC^2 \mid |x|>|y|>|x-y|>0 \} \\
\cO_3 &= \{(x,y,z)\in \bC^3 \mid |x|>|y|>|z|>|x-z|>|y-z|>|x-y| > 0 \}
\end{aligned} \]
These are non-empty, since they contain $(3,2)$ and $(7,6,4)$, respectively.

First, we describe an extension of module structure.

\begin{lem} \label{lem:extended-module-structure}
Let $V \subseteq T$ be an inclusion of non-negatively graded $C_2^0$-cofinite vertex operator algebras, let $U$ be a $V$-module, and let $\cY \in I\binom{U}{T,U}$ be a $V$-module intertwining operator that satisfies the following conditions:
\begin{enumerate}
\item $\cY$ factors through the inclusion $U((z)) \subseteq U\{z\}[\log z]$.
\item The restriction of $\cY$ to $V \otimes U \subseteq T \otimes U$ is equal to the $V$-action.
\item For any $t^1, t^2 \in T$, $w \in U$, $w' \in U'$, we have the equality
\[ \langle w', \cY(t^1,x)\cY(t^2,y)w\rangle = \langle w',\cY(Y(t^1,x-y)t^2,y)w\rangle \]
of holomorphic functions on $\cO_2 =\{(x,y)\in \bC^2 \mid |x|>|y|>|x-y|>0 \}$.
\end{enumerate}
Then $\cY$ yields a $T$-module structure on $U$.
\end{lem}
\begin{proof}
By Lemma \ref{lem:tensor-products-work-well} the holomorphic function given by the equality on $\cO_2$ extends uniquely to a multivalued holomorphic function $f_{w',t^1,t^2,w}$ on $\bC^2 \setminus \{(x,0), (0,y), (x,x) \}$ with regular singularities on the boundary divisors (including the line at infinity).  By switching $t^1$ with $t^2$ and $x$ with $y$, we obtain the equality
\[ \langle w', \cY(t^2,y)\cY(t^1,x)w\rangle = \langle w',\cY(Y(t^2,y-x)t^1,x)w\rangle\]
on $\{(x,y)\in \bC^2 \mid |y|>|x|>|x-y|>0 \}$, which extends uniquely to a holomorphic multivalued function $f_{w',t^2,t^1,w}$ on $\bC^2 \setminus \{(x,0), (0,y), (x,x) \}$ with regular singularities as before.

These two functions are related by the shearing operator $e^{(x-y)\frac{\pd}{\pd y}}$, which sends a holomorphic function $f(x-y,y)$ on $\{(x,y)\in \bC^2 \mid |y|>|x-y|>0 \}$ to $e^{(x-y)\frac{\pd}{\pd y}}f(x-y,y) = f(x-y,x)$ on $\{(x,y)\in \bC^2 \mid |x|>|x-y|>0 \}$.  On the intersection of these domains, we may compare $f_{w',t^1,t^2,w}$ with $f_{w',t^2,t^1,w}$, and we claim that they match.  To show this, we note that skew-symmetry for $Y$ and the $L(-1)$-derivative property for $\cY$ yield
\[ \begin{aligned}
\langle w',\cY(Y(t^1,x-y)t^2,y)w\rangle &= \langle w',\cY(e^{(x-y)L(-1)}Y(t^2,y-x)t^1,y)w\rangle \\
&= \langle w', e^{(x-y)\frac{\pd}{\pd y}} \cY(Y(t^2,y-x)t^1,y)w \rangle \\
&= e^{(x-y)\frac{\pd}{\pd y}} \langle w', \cY(Y(t^2,y-x)t^1,y)w \rangle \\
&= \langle w', \cY(Y(t^2,y-x)t^1,x)w \rangle.
\end{aligned} \]
We therefore find that $f_{w',t^1,t^2,w}$ has the following formal expansions:
\[ \begin{aligned}
\langle w', \cY(t^1,x)\cY(t^2,y)w\rangle &\in \bC((x))((y)) \\
\langle w',\cY(Y(t^1,x-y)t^2,y)w\rangle &\in \bC((y))((x-y)) \\
\langle w',\cY(Y(t^2,y-x)t^1,x)w\rangle &\in \bC((x))((x-y)) \\
\langle w', \cY(t^2,y)\cY(t^1,x)w\rangle &\in \bC((y))((x))
\end{aligned} \]
From these expansions, we see that the function $f_{w',t^1,t^2,w}$ has trivial formal monodromy on the boundary divisors, and thus descends to a single-valued meromorphic function on $\bP^2$ of the form $f_{w',t^1,t^2,w} \in \bC[x,y][x^{-1},y^{-1},(x-y)^{-1}]$.  This means that the operator $\cY$ satisfies the rationality, commutativity, and associativity conditions in \cite{FHL93} Proposition 4.5.1, so it yields a $T$-module structure on $U$.
\end{proof}

Second, we show that associativity is compatible with composition and iteration of intertwining operators.

\begin{lem} \label{lem:generalized-associativity}
Suppose $V$ is a non-negatively graded $C_2^0$-cofinite vertex operator algebra, and suppose we have $V$-modules $\{ U_i \}_{i=1}^{14}$ and the following intertwining operators:
\[ \begin{array}{cccc}
\cY_1 \in I\binom{U_4}{U_2,U_3} & \cY_2 \in I\binom{U_6}{U_1,U_4}
&\cY_3 \in I\binom{U_5}{U_1,U_2} &\cY_4 \in I\binom{U_6}{U_5,U_3} \\ \\
\cY_5 \in I\binom{U_1}{U_7,U_8} &\cY_6 \in I\binom{U_2}{U_9,U_{10}} 
&\cY_7 \in I\binom{U_3}{U_{11},U_{12}} &\cY_8 \in I\binom{U_{13}}{U_{14}, U_6} 
\end{array} \]
such that $\cY_1$, $\cY_2$, $\cY_3$, and $\cY_4$ satisfy
\begin{equation} \label{eq:basic-associativity}
\langle w'_6, \cY_2(u^1, x)\cY_1(u^2, y)u^3 \rangle = \langle w'_6,\cY_4(\cY_3(u^1, x-y)u^2, y)u^3 \rangle
\end{equation}
for all $u^1 \in U_1$, $u^2 \in U_2$, $u^3 \in U_3$, $w'_6 \in U'_6$ on $\cO_2$.  Then the following equalities of holomorphic functions on $\cO_3$ hold for all $u_i \in U_i$ and $w'_i \in U'_i$:
\[\begin{aligned}
\langle w'_{13}, &\cY_8(u^{14}, x) \cY_2(u^1, y)\cY_1(u^2, z) u^3 \rangle \\
&= \langle w'_{13}, \cY_8(u^{14}, x) \cY_4(\cY_3(u^1, y-z)u^2, z) u^3 \rangle \\
\langle w'_6, &\cY_2(u^1,x) \cY_1(\cY_6(u^9, y-z)u^{10}, z)u^3 \rangle \\
&= \langle w'_6, \cY_4(\cY_3(u^1, x-z)\cY_6(u^9, y-z)u^{10}, z) u^3 \rangle \\
\langle w'_6, &\cY_2(\cY_5(u^7, x-y) u^8, y)\cY_1(u^2, z)u^3 \rangle \\
&= \langle w'_6, \cY_4(\cY_3(\cY_5(u^7, x-y)u^8, y-z)u^2, z)u^3 \rangle \\
\langle w'_6, &\cY_1(u^1, x) \cY_2(u^2, y) \cY_7(u^{11}, z)u^{12} \rangle \\
&= \langle w'_6, \cY_4(\cY_3(u^1, x-y)u^2, y)\cY_7(u^{11}, z)u^{12} \rangle 
\end{aligned}\]
\end{lem}
\begin{proof}
By passing through the equivalence between $P(z)$ intertwining operators and intertwining operators, we find that all of the formal power series listed are pointwise absolutely convergent on $\cO_3$ by Proposition 12.7 of \cite{HLZ8}.  They yield holomorphic functions by the absolute uniform convergence on compact subsets, by a standard argument (e.g., it is straightforward to extend Lemma 7.7 of \cite{HLZ5} to multiple variables).

For the first equality, we fix $u^{14}$ and expand $\cY_8(u^{14},x)$ as a formal logarithmic power series, so that for each $(\alpha, k) \in \bQ \times \bN$, the $x^\alpha (\log x)^k$ term is a linear map $\phi_{\alpha,k}: U_6 \to U_{13}$ that translates generalized $L(0)$-eigenvalues by $\alpha$.  Let $w'_6 = w'_{13} \circ \phi_{\alpha,k}$.  This is an element in $U'_6$, so both sides are logarithmic power series in $x$ where the $x^\alpha (\log x)^k$ term has the form $\langle w'_6, \cY_2(u^1, y)\cY_1(u^2, z) u^3 \rangle$ on the left and $\langle w'_6, \cY_4(\cY_3(u^1, y-z)u^2, z) u^3 \rangle$ on the right.  By equation \eqref{eq:basic-associativity}, the coefficients are equal as long as $(y,z) \in \cO_2$, so we therefore have an equality of logarithmic power series in $x$ with coefficients in holomorphic functions on $\cO_2$.  Then for each point in $\cO_3$, we obtain an equality of absolutely convergent logarithmic power series in the single variable $x$.  By the convergence properties of these power series, we obtain equality of holomorphic functions.

The remaining cases follow similar arguments, where by isolating coefficients of $\cY_6$, $\cY_5$, and $\cY_7$, we obtain an equality of logarithmic power series with coefficients in holomorphic functions on $\cO_2$.  These yield equality of holomorphic functions on $\cO_3$ by the convergence properties we mentioned.
\end{proof}

\section{Projectivity of the fixed-point subalgebra}

If $V$ is a VOA, then a $V$-module $P$ is called projective  
if every $V$-module epimorphism $f:W\rightarrow P$ splits. 
The main purpose in this section is to prove the following theorem. 

\begin{uthm}
Let $T$ be a simple regular VOA with a nonsingular invariant bilinear form and 
$\sigma\in {\rm Aut}(T)$ of order $n \in \bZ_{\geq 1}$, then $T^\sigma$ is projective as a $T^\sigma$-module. 
\end{uthm}

To simplify the notation, we let $V = T^\sigma$ be the fixed-point subalgebra.
Viewing $T$ as a $\langle\sigma\rangle$-module, $T$ decomposes into $T=\oplus_{j=0}^{n-1} T^{(j)}$, 
where $T^{(j)}=\{v\in T\mid v^g=e^{2\pi ij/n}v \}$. By \cite{DM}, each $T^{(j)}$ is a simple $V$-module.

\subsection{Induced generalized logarithmic modules}

For each finitely generated $V$-module $D$, we consider the $V$-module:
\[ T\boxtimes_V D=\oplus_{j=0}^{n-1} T^{(j)}\boxtimes_V D. \]
We call this the induced module.  We note that it does not necessarily admit a $T$-module structure in the usual VOA sense, because the power series from the action may have fractional powers and logarithms.  

\begin{prop} \label{prop:induced-module-associativity}
There is an action of $T$ on $T\boxtimes_V D$ by an intertwining operator $\cY \in I\binom{T \boxtimes_V D}{T, T \boxtimes_V D}$ that satisfies the following compatibility conditions:
\begin{enumerate}
\item Identity: $\cY(\unit, z) = Id_{T \boxtimes_V D} z^0$
\item $L(-1)$-derivative: $\cY(L(-1)t,z) = \frac{\pd}{\pd z}\cY(t,z)$
\item Associativity: The equality $\langle w', \cY(t^1,x)\cY(t^2,y)w\rangle=\langle w',\cY(Y(t^1,x-y)t^2,y)w\rangle$ holds for all $t^1, t^2 \in T$, $w \in T \boxtimes_V D$, $w' \in (T \boxtimes_V D)'$, and $(x,y) \in \cO_2 \subset \bC^2$.
\end{enumerate}
\end{prop}
\begin{proof}
Set $W^{(j)} = T^{(j)}\boxtimes_V D$ and $W=\oplus_{j=0}^{p-1}W^{(j)} = T \boxtimes_V D$.
For each $T^{(i)}$, Corollary 9.30 of \cite{HLZ6} asserts that there is a unique $V$-module intertwining operator $\cY'_i \in I\binom{W}{T^{(i)},W}$ such that on $\cO_2$, we have the following equality:
\begin{equation}
\langle w',\cY'_i(t,x)\cY_{\boxtimes_{P(y)}}(t^1,y)d\rangle
=\langle w',\cY_{T,D}^{\boxtimes}(Y(t,x-y)t^1,y)d\rangle
\end{equation}
for any $t \in T^{(i)}$, $t^1\in T$, $w'\in W'$, and $d\in D$.  We choose an isomorphism $\phi: T \boxtimes_{P(y)} D \to T \boxtimes_V D$ such that $\cY_{T,D}^{\boxtimes} = \phi \circ \cY_{\boxtimes_{P(y)}}$, and define $\cY_i$ to be the composite $\cY^{P(y)}_i \circ \phi^{-1}$ for $i \neq 0$ and set $\cY_0$ to be the $V$-action.  Then by combining these, we obtain a $V$-module intertwining operator $\cY \in I\binom{W}{T,W}$ satisfying
\begin{equation} \label{eq:normalization-of-T-action}
\langle w',\cY(t,x) \cY_{T,D}^{\boxtimes}(t^1,y)d\rangle
=\langle w',\cY_{T,D}^{\boxtimes}(Y(t,x-y)t^1,y)d\rangle
\end{equation}
on $\cO_2$ for any $t, t^1\in T$, $w'\in W'$, and $d\in D$.  Because of our choice of $\cY_0$, our operator $\cY$ satisfies the identity condition, and the $L(-1)$-derivative condition follows from the definition of intertwining operator.


We now show that $\cY(t,x)$ satisfies associativity.  By repeatedly applying Lemma \ref{lem:generalized-associativity} to the equality \eqref{eq:normalization-of-T-action}, we obtain the following chain of equalities of holomorphic functions on the open domain $\cO_3$:
\[ \begin{aligned}
\langle w', \cY(t^1,x)\cY(t^2,y)\cY_{T,D}^{\boxtimes}(t^3,z)d\rangle
&= \langle w', \cY(t^1,x)\cY_{T,D}^{\boxtimes}(Y(t^2,y-z)t^3,z)d\rangle \\
&= \langle w', \cY_{T,D}^{\boxtimes}(Y(t^1,x-z)Y(t^2,y-z)t^3,z)d\rangle \\
&= \langle w', \cY_{T,D}^{\boxtimes}(Y(Y(t^1,x-y)t^2,y-z)t^3,z)d\rangle \\
&= \langle w', \cY(Y(t^1,x-y)t^2,y)\cY_{T,D}^{\boxtimes}(t^3,z)d\rangle
 \end{aligned} \]
for any $t^1,t^2,t^3\in T$, $w'\in W'$, and $d\in D$.  We therefore obtain an equality of formal power series in $z$ whose coefficients are holomorphic functions on $\cO_2$.  By isolating coefficients, the surjectivity of $\cY_{T,D}^{\boxtimes}$ implies the pointwise equality
\begin{equation} \label{eq:associativity-for-T-action}
\langle w', \cY(t^1,x)\cY(t^2,y)w\rangle=\langle w',\cY(Y(t^1,x-y)t^2,y)w\rangle
\end{equation}
on $\cO_2$ for any $t^1,t^2\in T$, $w'\in W'$, and $w \in W$.  Absolutely uniform convergence for the series on both sides yields equality of holomorphic functions, hence associativity.

\end{proof}

\subsection{Simple current property for eigenmodules}

As an application of the induced module, we prove the following theorem. 

\begin{thm}\label{thm:simple-current}
Let $T$ be a simple non-negatively graded $C_2$-cofinite VOA with nonsingular invariant form, and let $\sigma \in \Aut (T)$ of order $n \in \bZ_{\geq 1}$.  Then for any $j \in \bZ/n\bZ$, the $T^\sigma$-module $T^{(j)}$ is a simple current.
\end{thm}
\begin{proof}
We consider the induced module $W=T\boxtimes_V T^{(j)}$ of $T^{(j)}$, and set $W^{(h)}=T^{(h-j)}\boxtimes_V T^{(j)}$. 
As we explained before, $T$ has an action on $W$ by $\cY$ given in \eqref{eq:normalization-of-T-action}. 
By the universal property of $P(1)$ tensor products, there is a natural epimorphism $\phi: W=T \boxtimes_V T^{(j)} \to T$ such that $\phi(\cY^{\boxtimes}_{T,T^{(j)}}(t,z)t^j)=Y(t,z)t^j$ for $t\in T$, $t^j\in T^{(j)}$. 
Then by the definition of $\cY$, we have 
\[ \phi(\cY(v,z)w)=Y(v,z)\phi(w) \]
for $v\in T$ and $w\in W$. 
In particular, $W/\Ker\phi$ is isomorphic to $T$ as a $T$-module. 
We set $K^{(h)}:=W^{(h)}\cap \Ker\phi$, so $\Ker\phi=\oplus_{h=0}^{n-1}K^{(h)}$. 
Since $T^{(0)}\cong V$ and $W^{(j)}=T^{(0)}\boxtimes_V T^{(j)}\cong T^{(j)}$ is simple, we have $K^{(j)}=0$.

Suppose $W^{(0)}$ is not simple, i.e., $K^{(0)} \neq 0$.  By restricting the actions of $T$ on the induced module $W$ to $\Ker\phi$, we have the equality
\begin{equation} \label{eq:associativity-with-kernel}
\langle w', \cY(t^1,x)\cY(t^2,y)w\rangle=\langle w',\cY(Y(t^1,x-y)t^2,y)w\rangle
\end{equation}
on $\cO_2$ for any $t^1\in T^{(h-j)},t^2\in T^{(j-h)}$ and $w\in K^{(h)} $ and $w'\in (K^{(h)})'$.  Since $K^{(j)}=0$, we have $\cY(t^2,y)w=0$ and the left hand side of \eqref{eq:associativity-with-kernel} vanishes. 
On the other hand, $Y(t^1,x-y)t^2 \in V((x-y))$, so the restriction of $\cY$ in $\cY(Y(t^1,x-y)t^2,y)w$ is given by the module structure on $K^{(h)}$.  Furthermore, by the existence of a nondegenerate invariant form on $T$, some choice of $t^1$ and $t^2$ yields the vacuum as a coefficient of $Y(t^1,x-y)t^2$.  Thus, the right side of \eqref{eq:associativity-with-kernel} does not vanish identically, yielding a contradiction.

We conclude that $T^{(n-j)}\boxtimes_V T^{(j)}\cong V$.  For any simple module $D$, the associativity of fusion (Lemma \ref{lem:tensor-products-work-well}) implies
\[ T^{(n-j)}\boxtimes (T^{(j)}\boxtimes_V D)
\cong (T^{(n-j)}\boxtimes T^{(j)})\boxtimes_V D\cong V \boxtimes_V D\cong D. \]
Fusion with $T^{(j)}$ is therefore an invertible operation on isomorphism classes of irreducible $V$-modules, i.e., $T^{(j)}$ is simple current. 
\end{proof}

\begin{cor} \label{cor:irreducible-induced-modules}
If $B$ is a simple $V$-module, then $T \boxtimes_V B$ is isomorphic to the direct sum of $r$ copies of some irreducible $\sigma^k$-twisted $T$-module $W$ for some $k \in \bZ$ and some $r|n$.  In particular, any simple $V$-module is isomorphic to a direct summand of some irreducible twisted $T$-module.
\end{cor}
\begin{proof}
We first show that $T \boxtimes_V B$ is a twisted module.  

Since each $T^{(j)}$ is a simple current, the $V$-submodules $T^{(j)} \boxtimes_V B$ are irreducible, so the
intertwining operator $\cY_{T^{(1)},B}^\boxtimes$ has exponents in a single coset $s + \bZ$ for some $s \in \bC$.  By associativity, the intertwining operators $\cY_{T^{(j)},B}^\boxtimes$ then have exponents in $js + \bZ$ for all $j \in \bZ$, so $s \in \frac{1}{n} \bZ$, and the axioms for $\sigma^{ns}$-twisted modules follow immediately from the associativity of intertwining operators.

To decompose the twisted $T$-module $T \boxtimes_V B$, we first note that the simple current property of the modules $\{ T^{(i)} \}_{i=0}^{n-1}$ immediately implies the following:
\begin{enumerate}
\item The set of integers $k$ such that $T^{(k)} \boxtimes_V B \cong B$ has the form $\frac{n}{r}\bZ$ for some $r|n$.
\item The isomorphism classes of simple $V$-modules $\{ T^{(k)} \boxtimes_V B\}_{k \in \bZ}$ are periodic modulo $n/r$, and pairwise non-isomorphic between residue classes.
\end{enumerate}
For convenience, we will let $d = n/r$ for the remainder of this proof.  

We now claim that the action by $V$ on $B$ extends to one by $T^{\sigma^r} \cong \bigoplus_{k=0}^{r-1} T^{(kd)}$, so that $B$ is an irreducible $T^{\sigma^r}$-module.  By the simple current property of $T^{(id)}$ proved in Theorem \ref{thm:simple-current}, the intertwining operators $\cY^{id}: T^{(id)} \otimes B \to B((z))$ can only be chosen with constant ambiguity, so an arbitrary choice of nonzero $\cY^{(id)}$ will yield nonzero constants $\lambda_i$ satisfying
\[ \langle b', \cY^d(t^d, x) \cY^{id}(t^{id},y) b \rangle = \lambda_{i+1} \langle b', \cY^{id+d}(Y(t^d, x-y)t^{id}, y) b \rangle \]
on the domain $\cO_2$.  We will show that the operators $\cY^{id}$ can be rescaled so that $\lambda_i = 1$ for all $i$.  We tautologically have $\lambda_1 = 1$, and for each $1 < i < r$, we may successively replace each $\cY^{id}$ with $\lambda_i \cY^{id}$ to send our new $\lambda_i$ to one.  Then to eliminate $\lambda_0$, we choose an $r$-th root $\xi$ of $\lambda_0$ and replace each $\cY^{id}$ with $\xi^i \cY^{id}$.  By this choice of $\{ \cY^{id} \}_{i=0}^{r-1}$, the intertwining operators satisfy associativity, and we therefore obtain a $T^{\sigma^r}$-module structure on $B$.  We note that we also obtain a $T^{\sigma^r}$-module structure on $B$ by replacing each $\cY^{id}$ with $\zeta^i \cY^{id}$ for $\zeta$ any $r$-th root of unity.  We will call such a structure the $\zeta$-structure.

Choose a $V$-module decomposition $T^{\sigma^r} \boxtimes_V B \cong B^{\oplus r}$ such that the action of $T^{(d)}$ takes the $i$th copy of $B$ to power series with coefficients in the $i+1$st copy.  We wish to show that there is a $T^{\sigma^r}$-module decomposition $T^{\sigma^r} \boxtimes_V B \cong B^{\oplus r}$.  Consider the restriction of $\cY$, as defined in the proof of Theorem \ref{thm:simple-current}, to a map $T^{\sigma^r} \otimes (T^{\sigma^r} \boxtimes_V B) \to (T^{\sigma^r} \boxtimes_V B)((z))$.  Under our choice of $V$-module decomposition of $T^{\sigma^r} \boxtimes_V B$, we obtain have an $r \times r$ matrix of $V$-module intertwining operator maps $T^{(id)} \otimes B \to B((z))$ for each $i$, and the simple current property implies the entries are constant multiplies of $\cY^{id}$.  Let $\lambda^i_{k,\ell}$ denote the constant attached to the $(k,\ell)$ entry.  By our choice of $V$-module decomposition, $\lambda^i_{k,\ell}$ is nonzero when $k \equiv \ell + i \pmod{r}$, and zero otherwise.  By Proposition \ref{prop:induced-module-associativity}, $\cY$ satisfies associativity, so these constants satisfy the compatibility
\begin{equation} \label{eq:associativity-compatibility}
\lambda^i_{\ell+i,\ell} \lambda^j_{\ell+i+j,\ell+i} = \lambda^{i+j}_{\ell+i+j,\ell}
\end{equation}
where indices are given modulo $r$.

Consider the $r$ linearly independent $V$-module maps $f_j: B \to T^{\sigma^r} \boxtimes_V B \cong B^{\oplus r}$ given by
\[ f_j(b) = \left( b, \frac{\zeta^j}{\lambda^1_{2,1}}b, \cdots, \frac{\zeta^{j(r-1)}}{\lambda^{r-1}_{r,1}}b \right) \qquad 1 \leq j \leq r, \]
where $\zeta$ is a fixed primitive $r$-th root of unity.  We claim that $f_j$ is a $T^{\sigma^r}$-module map for the $\zeta^j$-structure on $B$.  To prove this, it suffices to check that the diagram
\[ \xymatrix{ T^{(id)} \otimes B \ar[r]^{\zeta^{ij} \cY^{id}} \ar[d]_{1 \otimes f_j} & B((z)) \ar[d]^{f_j((z))} \\
T^{(id)} \otimes B^{\oplus r} \ar[r]_{\cY} & (B^{\oplus r})((z)) } \]
commutes for all $i, j$.  Since the restriction of $\cY$ is given by the matrix $\lambda^i \cY^{id}$, we can check by matching the $k$th component of $1 \otimes f_j$ with the $i+k$th component of $f_j((z))$.  This amounts to checking the identity $\zeta^{ij} \frac{\zeta^{(k-1)j}}{\lambda^{k-1}_{k,1}} = \frac{\zeta^{(i+k-1)j}}{\lambda^{i+k-1}_{i+k,1}} \lambda^i_{i+k,k}$, which holds by the compatibility in equation \eqref{eq:associativity-compatibility}.  We therefore have a $T^{\sigma^r}$-module map
\[ (f_1,\ldots,f_r): B^{\oplus r} \to T^{\sigma^r} \boxtimes_V B \]
that is a $V$-module isomorphism, so it is promoted to a $T^{\sigma^r}$-module isomorphism.

We now have the twisted $T$-module decomposition 
\[ \begin{aligned}
T \boxtimes_V B &\cong T \boxtimes_{T^{\sigma^r}} (T^{\sigma^r} \boxtimes_V B) \\
&\cong T \boxtimes_{T^{\sigma^r}} (B^{\oplus r}) \\
&\cong (T \boxtimes_{T^{\sigma^r}} B)^{\oplus r}
\end{aligned} \]
To finish the proof, it suffices to show that $T \boxtimes_{T^{\sigma^r}} B$ is irreducible as a twisted $T$-module.  As a $V$-module (and in fact as a $T^{\sigma^r}$-module), we have the decomposition $T \boxtimes_{T^{\sigma^r}} B \cong \bigoplus_{k=0}^{d-1} T^{(k)} \boxtimes_V B$.  Any nonzero twisted $T$-submodule $W$ of $T \boxtimes_{T^{\sigma^r}} B$ is also a $V$-submodule, so it decomposes into a direct sum of simple $V$-modules, and there is some $k$ such that $T^{(k)} \boxtimes_V B$ is a direct summand of $W$.  The action of $T$ on $W$ restricts to nonzero intertwining operators $T^{(j)} \otimes (T^{(k)} \boxtimes_V B) \to W\{z\}$, so the simple current property of each $T^{(j)}$ implies $W$ then contains all of the $V$-module summands of $T \boxtimes_{T^{\sigma^r}} B$.  Thus, $W$ is isomorphic to $T \boxtimes_{T^{\sigma^r}} B$, which is therefore irreducible as a twisted $T$-module.
\end{proof}

\subsection{Projectivity}

We come to the main result of this section.

\begin{thm}\label{thm:OPR2}
Let $T$ be a simple regular VOA with a nonsingular bilinear form and 
$\sigma\in {\rm Aut}(T)$ of order $n$.  Then $T^\sigma$ is projective as a $T^\sigma$-module. 
\end{thm}
\begin{proof}
Suppose the assertion is false, and there is some non-split extension of $V$.  By Lemma \ref{lem:finite-length}, any $V$-module has finite length, so it suffices to consider the case where $0\to B \to D\xrightarrow{\rho} V \to 0$ is a non-split extension of $V$ by a simple $V$-module $B$.  By simplicity of $B$ and the non-split property, all weights of $D$ are integers.  For each character $j$, $T^{(j)}\boxtimes_V B$ is a simple $V$-module and the sequence of $V$-modules
\begin{equation} \label{eq:non-split-exact-sequence}
T^{(j)}\boxtimes_V B \xrightarrow{\rho_1} T^{(j)}\boxtimes_V D \xrightarrow{\rho_2} T^{(j)}
\boxtimes_V V \to 0
\end{equation}
is exact \cite{M4}.  Taking a direct sum over $j$, we have the following commutative diagram of $V$-modules with exact rows :
\[ \xymatrix{
 T\boxtimes_V B \ar[d]^{\tau_1} \ar[r]^{\rho_1} & T\boxtimes_V D \ar[d]^{\tau_2} \ar[r]^{\rho_2} & T\boxtimes_V V \ar[d]^{\tau_3} \ar[r] & 0 \\
B \ar[r] & D  \ar[r]^{\rho} & V \ar[r] & 0 } \]
where $\tau_i$ $(i=1,2,3)$ are all projections given by the direct factor $T^\sigma = V$ of $T$.  If $T^{(j)}\boxtimes_V B$ has non-integer weights, then $\rho_2:T^{(j)}\boxtimes_V D \to T^{(j)}$ 
has to split and so does $D \to V$, contradicting our initial assumption.  Therefore, we may assume all weights of $T\boxtimes_V D$ are integers.  Since $(T\boxtimes_V D)/\rho_1(T\boxtimes_VB)$ and $T\boxtimes_V B$ are finite direct sums of irreducible $T$-modules (the former by isomorphism with $T$, and the latter by Corollary \ref{cor:irreducible-induced-modules}), the Jordan blocks for the operator $L(0)$ on $T \boxtimes_V D$ have size at most 2, and
the power of $\log z$ in the expansion of $\cY$ defined by \eqref{eq:normalization-of-T-action} is at most one.  That is, for any $t \in T$, we have 
\[ \cY(t,z)=\cY^0(t,z)+\cY^1(t,z)\log z \]
with $\cY^0(t,z), \cY^1(t,z) \in \End(T\boxtimes_V D)[[z,z^{-1}]]$.  However, we may write $\cY^1(t,z)\in {\Hom}((T\boxtimes_V D)/\rho_1(T\boxtimes_V B), \rho_1(T\boxtimes_V B))[[z,z^{-1}]]$, since $\cY^1$ vanishes on inputs in $\rho_1(T\boxtimes_V B)$ and has coefficients that take values in $\rho_1(T\boxtimes_V B) \subseteq T \boxtimes_V D$.  Since $\cY(L(-1)t,z)=\frac{d}{dz}\cY(t,z)$, we have 
$\cY^0(L(-1)t,z)=\frac{d}{dz}\cY^0(t,z)+\frac{1}{z}\cY^1(t,z)$ and 
$\cY^1(L(-1)t,z)=\frac{d}{dz}\cY^1(t,z)$. In particular, by identifying $(T\boxtimes_V D)/\rho_1(T\boxtimes_V B)$ with $T$, we may say that $\cY^1$ is a $V$-module intertwining operator of type $\binom{T \boxtimes_V B}{T, T}$.

If $\cY^1=0$, then $\cY$ is an ordinary intertwining operator with integral powers satisfying associativity, so by Lemma \ref{lem:extended-module-structure}, it yields a $T$-module structure on $T\boxtimes_V D$ and the $V$-module surjection $\rho_2: T\boxtimes_V D\to T$ is promoted to a $T$-module surjection.  Because $T$ is a projective $T$-module by our assumption that $T$ is regular, $\rho_2$ then splits.  This induces a splitting of $\rho: V\boxtimes_V D \to V$, contradicting our assumption that the extension is not split.  Thus, we may assume $\cY^1 \neq 0$.  As we mentioned earlier in this proof, $T \boxtimes_V B$ is a sum of $r$ copies of some irreducible $T$-module for some $r|n$, so the existence of a nonzero $V$-module intertwining operator of type $\binom{T \boxtimes_V B}{T,T}$ implies the irreducible $T$-module in question decomposes as a $V$-module into a direct sum of finitely many simple currents $T^{(i)}$.  We therefore have $B \cong T^{(h)}$ for some $h$, and hence we have a $V$-module isomorphism $T \boxtimes_V B \cong T$.  The exact sequence from \eqref{eq:non-split-exact-sequence} then becomes
\[ T^{(h+j)} \to T^{(j)}\boxtimes_V D \to T^{(j)} \to 0 \]
and this is non-split for all $j$, since any splitting will yield a splitting of $D \to V$ by fusion with $T^{(n-j)}$.

We may restrict $\cY$ to $\cY_{i,j}: T^{(i)} \otimes (T^{(j)} \boxtimes_V D) \to (T^{(i+j)} \boxtimes_V D)((z))[\log z]$, and correspondingly decompose $\cY^0$ and $\cY^1$, so that $\cY^1_{i,j}$ is an intertwining operator of type $\binom{T^{(h+i+j)}}{T^{(i)},T^{(j)}}$.  By the simple current property, if $h \neq 0$, then $\cY^1_{i,j} = 0$ for all $i,j$ so $\cY^1$ vanishes, contradicting our assumption.  We therefore have $h=0$ and $B \cong V$.

By Lemma \ref{lem:extended-module-structure}, the restriction of $\cY^0$ to $\rho_1(T \boxtimes_V B)$ coincides with the induced $T$-module structure map, and $\cY^0$ projects to the $T$-module structure map on $(T\boxtimes_V D)/\rho_1(T\boxtimes_V B)$.

By the associativity property in Proposition \ref{prop:induced-module-associativity}, we have
\[ \begin{aligned}
\langle w', &\cY^0(t^i,x) \cY^0(t^k,y) w \rangle + \\
&\qquad + \langle w', \cY^0(t^i,x) \cY^1(t^k,y) w \rangle \log y + \langle w', \cY^1(t^i,x) \cY^0(t^k,y) w \rangle \log x \\
&= \langle w', \cY(t^i,x) \cY(t^k,y) w \rangle \\
&= \langle w', \cY(Y(t^i,x-y)t^k, y) w \rangle \\
&= \langle w', \cY^0(Y(t^i,x-y)t^k, y) w \rangle + \langle w', \cY^1(Y(t^i,x-y)t^k, y) w \rangle \log y
\end{aligned}\]
where the term with $\cY^1(t^i,x) \cY^1(t^k,y)$ is omitted on the left side because $\cY^1$ has coefficients taking values in $\rho_1(T\boxtimes_V B)$ and vanishes on inputs in $\rho_1(T\boxtimes_V B)$.  When $w \in \rho_1(T \boxtimes_V B)$, all of the logarithmic terms vanish, and we just get the associativity identity for the module structure.  Thus, the logarithmic terms only depend on the equivalence class of $w$ in the quotient module $(T\boxtimes_V D)/\rho_1(T\boxtimes_V B)$, and we get
\[ \begin{aligned}
\langle w', &\cY^0(t^i,x) \cY^1(t^k,y) w \rangle \log y + \langle w', \cY^1(t^i,x) \cY^0(t^k,y) w \rangle \log x \\
&= \langle w', \cY^1(Y(t^i,x-y)t^k, y) w \rangle \log y 
\end{aligned} \]
on $\cO_2$ for $w \in (T\boxtimes_V D)/\rho_1(T\boxtimes_V B)$ and $w' \in \rho_1(T \boxtimes_V B)'$.

Using our identification $\cY^0 = Y$ and $\cY^1_{i,j} = \lambda_{i,j} Y$, we have $\lambda_{0,j} = 0$, since $\cY_{0,j}$ was chosen to coincide with the $V$-action.  If we take $t^i \in T^{(j)}$, $t^k \in T^{(n-j)}$, and $w \in T^{(m)}$, then the right side of the equation has a factor of $\lambda_{0,m}$, and hence vanishes.  Thus, we find that $\lambda_{n-j,m} \langle w', Y(t^i,x) Y(t^k,y) w \rangle \log y = - \lambda_{j,n-j+m} \langle w', Y(t^i,x) Y(t^k,y) w \rangle \log x$ for all $(x,y) \in \cO_2$.  Since the ratio of $\log y$ and $\log x$ is not constant on $\cO_2$, both sides must vanish, and $\lambda_{j,m} = 0$ for all $(j,m)$.  We conclude that $\cY^1$ vanishes, contradicting our earlier reduction.

Thus, $V$ is projective as a $V$-module.
\end{proof}

\section{Regularity}
In order to prove the main theorem, we will combine a result in \cite{M04} 
about the modular invariance of characters with a method treating genus-one two-point correlation functions 
introduced in \cite{HD}. We first recall the result about the modular invariance property.

\subsection{Modular invariance}

For a $V$-module $U$, we consider the trace function 
\[ Z_U(v,\tau):=x^{\wt(v)}\Tr_{U} Y^U(v,x)q_{\tau}^{L(0)-c/24} \]
for $\tau$ in the complex upper half-plane $\cH$ and $v\in V_{\wt(v)}$, where $q_z$ denotes $e^{2\pi iz}$ for any suitable input $z$. We note that $Z_U(v,\tau)$ does not depend on $x$.

\begin{thm} \label{thm:modular}
If $V$ is a simple $C_2$-cofinite vertex operator algebra with non-negative $L(0)$-spectrum, and $U^1,\ldots,U^r$ represent the isomorphism classes of irreducible $V$-modules, then there are pseudo-trace functions $\widehat{\Tr}_{r+1}\ldots,\widehat{\Tr}_{r+d}$ and complex numbers $s_{j,k}$ for $j = 1,\ldots,r$ and $k = 1, \ldots, r+d$, such that the following equality of functions on $V \times \cH$ holds:
\[ \tau^{-\wt [v]} Z_{U^j}(v,-1/\tau) = \sum_{k=1}^r s_{j,k} Z_{U^k}(v,\tau) + \sum_{k = r+1}^{r+d} s_{j,k} \widehat{\Tr}_k o(v) q^{L(0)-c/24} \]
Here, $v\in V_{[\wt[v]]}$, where the bracket weight $V_{[m]}$ is defined by Zhu's modified VOA structure $(V,Y[,z],\unit,\widetilde{\omega})$ on $V$ given by $Y[v,z]:=Y(v,q_z-1)q^{\wt(v)}$ and conformal element $\widetilde{\omega} = (2\pi i)^2(\omega-\frac{c}{24}\unit)$. 
\end{thm}
\begin{proof}
This is immediate from Theorem 5.5 of \cite{M04}: $Z_{U^j}$ is an element of the one-point genus one function space $\cC_1(V)$, which is $SL_2(\bZ)$-invariant and spanned by traces on irreducible modules and pseudo-traces.  However, in the proof of this theorem, there is an erroneous claim in the beginning of the pseudo-trace section (pointed out in Remark 3.5.2 of \cite{AN}), namely that if $V$ is a non-negatively graded $C_2$-cofinite VOA and $W^i$ is an irreducible $V$-module, then the homogeneous subspace $W^i(m)$ is nonzero for all sufficiently large $m$.  A counterexample is given by the one-dimensional module $\bC \unit$ for the $\cW_{2,3}$ VOA of central charge zero.  This VOA was shown to be $C_2$-cofinite in \cite{AM}.  We will now show that if the erroneous claim does not hold for an irreducible module $W^i$, then the corresponding trace function is constant, and the statement of the theorem still holds.

Suppose there is an infinite set of positive integers $m$ such that $W^i(m) = 0$.  Since any lowest-weight representation of the Virasoro algebra generated by a single vector has at most finitely many gaps, this implies the Virasoro action is trivial.  Therefore, the module $W^i$ is equal to its lowest-weight space $W^i(0)$, and $W^i$ is then an irreducible $A_0(V)$-module, hence finite dimensional.  Let $\Ann(W^i) = \{ v \in V | o(v)W^i = 0 \}$ be the annihilator of $W^i$, where we set $o(v) = v_{\wt(v) -1}$ for homogeneous elements $v \in V$ and extend it to a linear map $o: V \to \End_{Vect}(W^i)$.  Because $W^i$ is finite dimensional, $\dim (V/\Ann(W^i)) < \infty$.  Furthermore, $\Ann(W^i)$ is an ideal, because for $u \in V$, $v \in \Ann(W^i)$, and $w \in W^i$, the Borcherds identity \cite{B} presents any $(u_m v)_n w$ as a linear combination of terms of the form $u_{m-i} (v_{n+i} w)$ and $v_{m+n-i} (u_i w)$ with $i \geq 0$, all of which vanish.  Traces then arise from the induced action of the quotient vertex operator algebra $V/\Ann(W^i)$.  We note that in the case of interest to us, where $V$ is simple and infinite dimensional, we obtain a contradiction, i.e., a finite dimensional irreducible module $W^i$ cannot exist.  However, in general, this quotient is a finite dimensional commutative ring with trivial weights, so an irreducible module is necessarily one dimensional, and the traces are just constants.  In fact, a short computation shows that the space of these constant trace functions is isomorphic to $V/P(V) \cong V_0/(V_0 \cap P(V))$, where $P(V)$ is the subspace of $V$ spanned by $\{u_k v | u,v \in V, k \geq 0\}$.
\end{proof}

\begin{rem}
From the previous argument, it is straightforward to see that formula for the dimension of $\cC_1(V)$ in Theorem 5.5 of \cite{M04} needs to be revised from $\dim A_n(V)/[A_n(V),A_n(V)] - \dim A_{n-1}(V)/[A_{n-1}(V),A_{n-1}(V)]$ to
\[ \dim A_n(V)/[A_n(V),A_n(V)] - \dim A_{n-1}(V)/[A_{n-1}(V),A_{n-1}(V)] + \dim V/P(V)\]
to account for trace functions on one-dimensional irreducible modules.
\end{rem}

We may extend trace functions to virtual modules with complex coefficients: for a linear combination $U=\sum_{i\in D} \xi_iU^i$ of simple modules $U^i$, we will abuse the notation $Z_U(v,\tau)$ and $s_{U,j}$ to denote $\sum_{i\in D} \xi_iZ_{U^i}(v,\tau)$ and $\sum_{i\in D} \xi_i s_{i,j}$, respectively. 

We won't explain the precise definition of pseudo-trace functions, because we are concerned with only the cases 
where pseudo-trace functions don't appear. 

\begin{defn}
If $\tau^{-\wt[v]}Z_{U}(v,-1/\tau)$ is a linear combination of trace functions on $V$-modules for 
all $v\in V$, then we say that the virtual module $U$ is of type NPT. 
\end{defn}

One point about pseudo-trace functions that we would like to mention is that for any pseudo-trace function $\widehat{\Tr}_k$, we have $\widehat{\Tr}_k o(\unit)q^{L(0)}=0$.

\subsection{Two-point genus one functions}

We introduce a family of composition-invertible power series defined in \cite{HD}, together with their actions on modules for the Virasoro algebra.  We then extend some of Huang's results to the case of intertwining operators with logarithmic terms.

Let $A_j, j \in \bN$ be the complex numbers defined by 
\[ \frac{1}{2\pi i}(e^{2\pi i y}-1) = \left( \exp \left( - \sum_{j \in \bN} A_j y^{j+1} \frac{\partial}{\partial y} \right) \right) y. \]
For any $V$-module $W$, we denote the operator $\sum_{j \in \bN} a_j L(j)$ on $W$ by $L_+(A)$.  For convenience, we introduce the operator
 \[ \cU(x) = (2 \pi i x)^{L(0)}e^{-L_+(A)} \in (\End W)\{x\} \]
 where $(2 \pi i)^{L(0)} = e^{(\log 2\pi + i \pi/2)L(0)}$.  We have $\cU(1) = (2\pi i)^{L(0)}e^{-L_+(A)}$ and $\cU(q_x) = q_x^{L(0)} \cU(1)$ as concrete examples.  As Huang has shown in [4, Proposition 1.2], $\cU(1)\omega = (2\pi i)^2(\omega - \frac{c}{24}\unit)$ and 
\[ \cU(1) \cY(w,x) \cU(1)^{-1} = \cY(\cU(q_x)w,q_x-1) \]
for an intertwining operator $\cY$.  Namely, we have 
\[ \cU(1) \cY(w,x)\cU(1)^{-1} = \cY[\cU(1)w,x] \]
where $\cY[w,x] = \cY(w,q_x-1)q_x^{\wt(w)}$ for homogeneous elements $w$ defines the modified intertwining operator for the modified VOA $(V, Y[,], \unit, \tilde{\omega} = \cU(1)\omega)$ in \cite{Zh}.  We see that
\[ \cU(1): \bigoplus_{U \in \textrm{mod}(V)} U\{x\} \to \bigoplus_{U \in \textrm{mod}(V)} U\{x\} \]
(where we take a sum over isomorphism classes to avoid set-theoretic problems) gives an isomorphism between $(V,Y, \unit,\omega)$-intertwining operators by $W$ and $(V, Y[,], \unit, \tilde{\omega})$-intertwining operators by $W$ via $\cU(1): W\{x\} \to W\{x\}$.

\begin{lem} \label{lem:u1-weight-shift}
For $u \in W_n$, $\cU(1)w \in W_{[n]}$.
\end{lem}
\begin{proof}
The $x^{-2}$-coefficients of both sides of $\cU(1)Y(\omega,x)u = Y[\cU(1)\omega,x]\cU(1)u$ yield $\wt(u)\cU(1)u = L[0]\cU(1)u$.
\end{proof}

The following equations are proved for non-logarithmic intertwining operators in \cite{HD}.  We will prove the same statement for every intertwining operator even if it has logarithmic terms.

\begin{lem} \label{lem:generalization-of-Huang}
Let $B, U, T, P$ be $V$-modules.  For $\cY \in I\binom{T}{B,U}$, $\cY_1 \in I\binom{P}{T,P}$, $u \in U$, and $b \in B$, we have the following formal equalities:
\begin{enumerate}
\item $e^{\tau L(0)}\cY(b,z)u = \cY(e^{\tau L(0)}b, e^\tau z)e^{\tau L(0)}u$ in $T \{z,e^\tau\}[y,\tau]$, where $y = \log z$.
\item $q^{L(0)}\cY(\cU(q_y)b,q_y) = \cY(\cU(q_yq)b, q_y q)q^{L(0)}$ in $(\Hom (U,T)) \{q_y, q \}[y,\tau]$
\item $\cY_1(\cY(\cU(q_y)b,q_y-q_x)\cU(q_x)u,q_x) = \cY_1(\cU(q_x)\cY(b, y-x)u, q_x)$ in $(\End P) \{q_x, y-x \}[x, \log(y-x)]$
\end{enumerate}
\end{lem}
\begin{proof}
Set $\cY(b,z) = \sum_{h=0}^K \sum_{n \in \bC} b_{(h,n)} z^{-n-1} y^h$.  From the $L(-1)$ derivative formula $\cY(L(-1)b,z) = \frac{d}{dz}\cY(b,z)$, we have
\[ \begin{aligned}
{[}L(0),b_{(h,n)}] u &= (L(-1)b)_{(h,n+1)} + (L(0)b)_{(h,n)} \\
&= (-n-1)b_{(h,n)} + (h+1)b_{(h+1,n)} + (L(0)b)_{(h,n)}.
\end{aligned} \]
Using the notation $(\alpha \otimes \beta)b_{(h,n)}u = (\alpha b)_{(h,n)} \beta u$, and $\Omega = L(0) \otimes 1 + 1 \otimes L(0)$, we have
\[ \begin{aligned}
L(0)(b_{(h,n)}u) &= (-n-1 + \Omega)b_{(h,n)}u + (h+1)b_{(h+1,n)}u \\
L(0)^m(b_{(h,n)}u) &= \sum_{j=0}^m \binom{m}{j} (-n-1 + \Omega)^{m-j} (h+1)\cdots(h+j) b_{(h+j,n)}u 
\end{aligned} \]
for $m \geq 1$, where $(h+1)\cdots (h+j) = 1$ for $j=0$.  Therefore, we obtain
\[ \begin{aligned}
e^{\tau L(0)} & \cY(b,z)u = \sum_{n \in \bC} e^{\tau L(0)} \left( \sum_{h=0}^K b_{(h,n)} uy^h \right) z^{-n-1} \\
&= \sum_{n \in \bC} \sum_{m=0}^\infty \sum_{h=0}^\infty \frac{L(0)^m \tau^m}{m!} (b_{(h,n)}uy^h z^{-n-1}) \\
&= \sum_{n \in \bC} \sum_{m=0}^\infty \sum_{h=0}^\infty \sum_{j=0}^m \frac{\tau^m}{m!} \binom{m}{j} (\Omega -n - 1)^{m-j}(h+1)\cdots(h+j) b_{(h+j,n)}uy^h z^{-n-1}.
\end{aligned}\]
By replacing $h+j$ and $m=j$ by $k$ and $i$, respectively, $e^{\tau L(0)} \cY(b,z)u$ equals
\[ \begin{aligned}
\sum_{n \in \bC} \sum_{k=0}^\infty &\sum_{i=0}^\infty \sum_{j=0}^k \frac{\tau^i (\Omega -n - 1)^i}{i!} \frac{1}{j!} (k-j+1)\cdots(k) \tau^j y^{k-j} b_{(k,n)}u z^{-n-1} \\
&= \sum_{n \in \bC} \sum_{k=0}^\infty e^{\tau(\Omega -n - 1)} b_{(k,n)}u (y+\tau)^k z^{-n-1} \\
&= \sum_{n \in \bC} \sum_{k=0}^\infty (e^{\tau L(0)} b)_{(k,n)} (e^{\tau L(0)}u) (y+\tau)^k e^{\tau(-n-1)} z^{-n-1} \\
&= \cY(e^{\tau L(0)} b, e^\tau z) (e^{\tau L(0)} u)
\end{aligned} \]
so the first equation is proved.  Replacing $\tau$ with $2 \pi i \tau$, $y$ with $2 \pi i y$, and $b$ with $\cU(q_y)b$ yields the second equation.  The third equation follows from the second equation, together with the following:
\[ \begin{aligned}
\cU(q_x)\cY(b,y-x) &= q_x^{L(0)}\cU(1) \cY(b,y-x) \\
&= q_x^{L(0)} \cY(\cU(q_{y-x})b, q_{y-x} - 1)\cU(1) \\
&= \cY(\cU(q_{y-x}q_x)b, q_x(q_{y-x}-1))q_x^{L(0)}\cU(1) \\
&= \cY(\cU(q_y)b, q_y - q_x)\cU(q_x)
\end{aligned} \]
\end{proof}

We now consider genus-one two-point correlation functions introduced by Huang in Proposition 1.4 of \cite{HD} (see also \cite{MS}).  From now on, to simplify the notation, we assume that $W$ is a simple $V$-module with integer weights.  We then have a canonical homomorphism $\pi_{W}:W\boxtimes W' \to V\, (\cong V')$
satisfying 
\[ \lim_{z\to 0}z^{\wt(w)+\wt(w')}\pi_{W}(\cY^{\boxtimes}_{W,W'}(w,z)w')=\langle w,w'\rangle\unit \]
for $w\in W$ and $w'\in W'$, where $\langle\cdot,\cdot \rangle$ is the contragradient pairing on $W \times W'$.  We set $\cY_W:=\pi_W\cY^{\boxtimes}_{W,W'}$, and warn the reader against forgetting about the $\pi_W$. 

\begin{defn}
Let $W$ be a simple $V$-module with integer weights, and let $U$ be a simple $V$-module.  We define the formal power series
\[ Z_U(w,w',x-y,\tau) := \Tr_U Y(\cU(q_y) \cY_W(w,x-y)w',q_y)q^{L(0)-c/24}. \]
This is an element of $\bC\{q_y,q\}((x-y))[y,\tau]$
\end{defn}

In \cite{HD}, Huang showed that this power series is absolutely uniformly convergent on compact sets in the region $\cO_4 = \{ (x,y,\tau) \in \bC^2 \times \cH | |q_{x-y}| > 1 > |q_{x-y}-1| > 0 \}$.  While Huang shows in Proposition 5.4 of \textit{loc. cit.} that this function satisfies a differential equation with regular singularities at specified loci, and therefore can be analytically continued to a multivalued function on a dense open subset of $\bC^2 \times \cH$, we will restrict our view to $\cO_4$ (which is an infinite disjoint union of contractible open sets).  This allows us to avoid the subtleties of showing that multivalued functions are equal.  By the third claim of Lemma \ref{lem:generalization-of-Huang}, this function is equal to 
\[ \Tr_U Y(\cY(\cU(q_y q_{x-y})w, q_y(q_{x-y}-1))\cU(q_y)w',q_y)q^{L(0)-c/24} \]
and we will view it as a formal power series in $q_{x-y}$ whose coefficients are analytic functions of $\tau \in \cH$.  A priori, it is also a formal power series in $q_y$, but by the following lemma, we can ignore that variable.

\begin{lem}
$\Tr_U Y(\cU(q_y) \cY_W(w,x-y)w',q_y)q^{L(0)-c/24}$ does not depend on $q_y$.
\end{lem}
\begin{proof}
Set $\cU(1)v = \sum_r v^{(r)}$, where $v^{(r)} \in V_r$.  Then the grade-preserving part of
\[ Y(\cU(q_y)v,q_y) = \sum_m q_y^{L(0)} \left( \sum v^{(r)} \right)_m q_y^{-m-1} = \sum_m \sum_r q_y^r (v^{(r)})_m q_y^{-m-1} \]
is $\sum_r q_y^r v_{r-1}^{(r)} q_y^r = \sum_r v_{r-1}^{(r)}$ which does not depend on $q_y$.
\end{proof}

\begin{rem} \label{rem:modular}
As we have seen, the $\cU$ operators allow us to switch between ordinary and modified vertex operator structures.  Here, we note that $\cU(q_y)$ can be used to clarify the $S$-transformation formula.  When we consider the $S$-transformation $\tau \mapsto -1/\tau$, we have to consider the bracket weight $\wt[v]$ of $v \in V$.  However, the coefficients $w_m w'$ of $\cY_W(w,x-y)w'$ satisfy $\wt(w_m w') = \wt(w) + \wt(w') - m - 1$ while $\cU(1) w_m w' \in V_{[\wt(w) + \wt(w') - m - 1]}$.  That is, for each $m \in \bC$, the coefficient of $(x-y)^m$ in $Z_U$ is a trace function on $U$ with the grade-preserving operator of $\cU(1) w_m w' \in V_{[\wt(w) + \wt(w') - m - 1]}$.  Therefore, the $S$-transformation of $Z_U$ is
\[ \tau^{\wt(w) + \wt(w')} \Tr_U Y(\cU(q_y)\cY_W(w, \frac{x-y}{\tau})w', q_y)q_{-1/\tau}^{L(0)-c/24}. \]
Then, since $w_m w' \in V_{\wt(w) + \wt(w') - m - 1}$, Lemma \ref{lem:u1-weight-shift} and Theorem \ref{thm:modular} imply
\[ \begin{aligned}
\Tr_{U^j}\tau^{-\wt(w) - \wt(w')}&Y(\cU(q_y) \cY_W(w,\frac{x-y}{\tau})w',q_y) q_{-1/\tau}^{L(0)-c/24} \\
&= \sum_{k=1}^r s_{jk} \Tr_{U^k}Y(\cU(q_y)\cY_W(w,x-y)w',q_y) q^{L(0)-c/24} \\ 
& \, + \sum_{k=r+1}^{r+d} s_{jk} \hat{\Tr}_k Y(\cU(q_y)\cY_W(w,x-y)w',q_y) q^{L(0)-c/24}
\end{aligned} \]
by viewing both sides as formal power series in $x-y$ with coefficients in analytic functions of $\tau$.
\end{rem}

\begin{rem} \label{rem:associativity}
Since the inequalities $|q_x| > |q_y|$ and $1 > |q_{x-y} - 1|$ hold in the region $\cO_4$, the second associativity claim of Corollary 9.30 in \cite{HLZ6} implies there is an intertwining operator $\cY_1 \in I\binom{U}{W, W' \boxtimes U}$ such that
\[ \langle u', Y(\cU(q_y) \cY_W(w,x-y)w',q_y)u \rangle = \langle u', \cY_1(\cU(q_x)w,q_x) \cY^\boxtimes(\cU(q_y)w',q_y)u \rangle \]
for $u \in U$, $u' \in U'$, $w \in W$, and $w' \in W'$.  In particular, as a formal power series of $q$ with the coefficients in analytic functions on $x$ and $y$, we have
\[ \begin{aligned}
\Tr_U &Y^U(\cU(q_y) \cY_W(w,x-y)w',q_y) q^{L(0)-c/24} = \\
&= \Tr_U \cY_1(\cU(q_x)w,q_x) \cY^\boxtimes(\cU(q_y)w',q_y)q^{L(0)-c/24} 
\end{aligned} \]
in $\cO_4$.  By Lemma \ref{lem:tensor-products-work-well}, the coefficient functions in $x$ and $y$ attached to a power of $q$ satisfy a certain regular singular differential equation with known singularities, so the above functions are all holomorphic on $\cO_4$.
\end{rem}

\begin{lem} \label{lem:linear-independence}
The set $\{ Z_U(w,w',x-y,\tau) \}_{U \in \Irr(V)}$ forms a linearly independent collection, when viewed as maps from $W \boxtimes W'$ to either holomorphic functions on $\cO_4$ or formal power series.
\end{lem}
\begin{proof}
By Proposition 4.23 of \cite{HLZ3}, $V$ is spanned by the coefficients of $\cY_W$, so it suffices to show linear independence of traces of $V$ on simple modules $U$.  This in turn follows from a straightforward adaptation of the arguments in the beginning of the proof of Theorem 5.3.1 of \cite{Zh}.  Essentially, if we restrict the trace functions to highest-weight subspaces, we get traces of zero-mode operators on distinct irreducible $A_0(V)$-modules, and we do not need semisimplicity to conclude their linear independence.
\end{proof}

Since $\Tr_U o(L[-1]v) = 0$ and $Z_U$ does not depend on $q_y$, we can extend the skew-symmetry property of products to the following skew-symmetry property of genus-one two-point correlation functions.

\begin{lem} \label{lem:skew-symmetry}If the weights of $W$ are integers, then we have 
\[ Z_U(w,w',y-x,\tau) = Z_U(w',w,x-y,\tau) \]
\end{lem}
\begin{proof}
By direct calculation, we have
\[ \begin{aligned}
Z_U(w,w',y-x,\tau) &= \Tr_U Y(\cU(q_x) \cY_W(w,y-x)w',q_x)q^{L(0)-c/24} \\
&= \Tr_U Y(\cU(q_x) e^{L(-1)(y-x)}\cY_{W'}(w',x-y)w,q_x)q^{L(0)-c/24} \\
&= \Tr_U Y(e^{L[-1](y-x)}\cU(q_x) \cY_{W'}(w',x-y)w,q_x)q^{L(0)-c/24} \\
 &= \Tr_U Y(\cU(q_y) \cY_{W'}(w',x-y)w,q_y)q^{L(0)-c/24}
\end{aligned} \]
We note that the integral weight property of $W$ is used in the second line.
\end{proof}

\subsection{Moore-Seiberg-Huang argument}

We now introduce a method which we call the Moore-Seiberg-Huang argument.  We wish to consider the behavior of two-point genus 1 trace functions under shifts along generating 1-cycles of an elliptic curve.  That is, we consider the formal $\bQ$-power series in $q_x$, $q_y$, and $q$:
\begin{enumerate}
\item $\Tr_P Y^P(\cU(q_y) \cY_{W',W}^V(w',x-y-1)w,q_y)q^{L(0)-c/24}$
\item $\Tr_U Y^U(\cU(q_x) \cY_{W,W'}^V(w,y-x-\tau)w',q_x)q^{L(0)-c/24}$
\end{enumerate}

From shifts by 1, we will simply pick up a phase.  From shifts by $\tau$, we will pass from a trace involving the action of $V$ on a (virtual) $V$-module to a trace involving intertwining operators from $W \boxtimes W'$.  By combining these shifts with modularity results via the $S$ transformation, we obtain strong compatibility information leading to a rigidity result.

For the first trace function (1), we have the following lemma

\begin{lem} \label{lem:translate-by-one}
Let $P$ be a simple $V$-module with weights in $p + \bN$.  Assume the weights of $W' \boxtimes P$ are in $r + \bN$.  Then
\[ Z_P(w,w',x-y-1,\tau) = e^{2\pi i(p-r)} Z_P(w,w',x-y,\tau) \]
\end{lem}
\begin{proof}
By Corollary 9.30 in \cite{HLZ6} (specifically, the second claim), there is an intertwining operator $\cY^1 \in I\binom{P}{W,W' \boxtimes P}$ such that
\[ \langle u', Y(\cY_W(w,x-y)w',y)u \rangle = \langle u', \cY^1(w,x) \cY_{W',P}^\boxtimes (w',y)u\rangle\]
 on $\cO_2$ for all $u' \in P'$ and $u \in P$.  Let $\{u_k | k \in B(P) \}$ be a homogeneous basis and $\{u'_k | k \in B(P) \}$ its dual basis.  Then since
\[ Z_P(w, w', x-y-1, \tau) = \sum_{k \in B(P)} \langle u'_k, \Tr_P Y(\cU(q_y) \cY_W(w,x-y-1)w',q_y)q^{\wt(u)-c/24} u_k \rangle\]
we use the expansion of $Z_P$ in Remark \ref{rem:associativity} to get:
\[ \begin{aligned}
Z_P(w, w', x-y-1, \tau) &= \Tr_P \cY^1(\cU(q_x)w,q_x) \cY_{W',P}^\boxtimes (\cU(q_y e^{2\pi i})w',q_y e^{2\pi i}) q^{L(0)-c/24} \\
&= \Tr_P \cY^1(\cU(q_x)w,q_x) e^{2\pi i (p-r)} \cY_{W',P}^\boxtimes (\cU(q_y)w',q_y) q^{L(0)-c/24} \\
&= e^{2\pi i (p-r)} \Tr_P \cY^1(\cU(q_x)w,q_x) \cY_{W',P}^\boxtimes (\cU(q_y)w',q_y) q^{L(0)-c/24} \\
&=  e^{2\pi i(p-r)} Z_P(w,w',x-y,\tau).
\end{aligned}\]
The second equality follows from the first claim of Lemma \ref{lem:generalization-of-Huang} with the substitution $\tau = e^{2\pi i}$.  While the Lemma asserted an equality of formal series in $\tau$, we also obtain an equality of holomorphic functions.
\end{proof}

For the second trace function (2), we have the following lemma

\begin{lem} \label{lem:translate-by-tau}
Let $U$ be a simple $V$-module.  Then as a formal $\bQ$ power series of $q_x$, $q_y$, and $q$, $Z_U(w,w',x-y-\tau,\tau)$ is a linear sum of trace functions of $W' \boxtimes W$ on $V$-modules $W' \boxtimes U$.  Namely, there is $\cY_U^2 
\in I\binom{W' \boxtimes U}{W' \boxtimes W, W' \boxtimes U}$ such that 
\[ Z_U(w,w',x-y-\tau,\tau) = \Tr_{W' \boxtimes U} \cY_U^2(\cU(q_x) \cY_{W',W}^\boxtimes (w',y-x)w, q_x) q^{L(0)-c/24}. \]
\end{lem}
\begin{proof}
By Corollary 9.30 in \cite{HLZ6} (in particular the second claim), there is an intertwining operator $\cY_U^1 \in I\binom{U}{W, W' \boxtimes U}$ such that
\[ \langle u', Y^U(\cY_W(w,x-y)w',y)u\rangle=\langle u', \cY_U^1(w,x)\cY_{W',U}^{\boxtimes}(w',y)u\rangle \]
and so we have 
\[ \begin{aligned}
Z_U(w,w',x-y-\tau,\tau) &= \Tr_U \cY_U^1(\cU(q_x)w, q_x) \cY_{W',U}^\boxtimes(\cU(q_y q)w',q_y q) q^{L(0)-c/24} \\
&= \Tr_U \cY_U^1(\cU(q_x)w, q_x) q^{L(0)-c/24}  \cY_{W',U}^\boxtimes(\cU(q_y)w',q_y) \\
&= \Tr_{W' \boxtimes U} \cY_{W',U}^\boxtimes(\cU(q_y)w',q_y) \cY_U^1(\cU(q_x)w, q_x) q^{L(0)-c/24}, 
\end{aligned}\]
where we use the second claim of Lemma \ref{lem:generalization-of-Huang} to obtain the second line, and the cyclic symmetry of traces to obtain the third line.  We apply the first associativity claim of Theorem 9.30 in \cite{HLZ6} to see that there is an intertwining operator $\cY_U^2 \in I\binom{W' \boxtimes U}{W' \boxtimes W, W' \boxtimes U}$ such that
\[ Z_U(w,w',x-y-\tau,\tau) = \Tr_{W' \boxtimes U} \cY_U^2\left(\cU(q_x) \cY_{W',W}^\boxtimes(w',y-x)w, q_x\right) q^{L(0)-c/24} \]
as we desired.
\end{proof}

We write $Z_{W' \boxtimes U}(\cY_U^2, w', w, y-x, \tau)$ to denote the right side of the equation in the above lemma.
In order to use both equations (1) and (2), we consider some niceness conditions for modules and virtual modules.

\begin{defn}
Suppose $V$ has non-negative $L(0)$-spectrum, and let $\{ U^i \}_{i \in D}$ be a collection of $V$-modules, and let $U = \sum_{i \in D} x_i U^i$ be a virtual module with complex coefficients.
\begin{enumerate}
\item[NPT] We say that $U$ satisfies the property NPT (for ``no pseudo-trace'') if the $S$-transformation of the trace function on $U$ is a linear combination of trace functions on $V$-modules.  Equivalently, the decomposition of $\tau^{\wt[v]}Z_U(v,-1/\tau)$ given in Theorem \ref{thm:modular} has no contributions from pseudo-traces for any $v \in V$.
\item[SW] We say that a module $U^i$ satisfies the property SW (single weight) with respect to an irreducible $V$-module $W$ if the generalized eigenvalues of $L(0)$ on $W \boxtimes U^i$ lie in a single coset of $\bZ$.  We say that the collection $\{ U_i \}_{ i \in D}$ satisfies SW if each $U^i$ satisfies SW.
\end{enumerate}
\end{defn}

\begin{lem} \label{lem:chain-of-equalities}
Suppose $\{ U^i \}_{i \in D}$ satisfies SW with respect to $W$, and $U = \sum_{i \in D} x_i U^i$ satisfies property NPT.  Then
\[ \begin{aligned}
\sum_{i \in D} &x_i \sum_j s_{i,j} Z_{W \boxtimes U^j}(\cY_{U^j}^2, w, w', x-y, \tau) \\
&= \sum_{i \in D} x_i \kappa_i \left( \sum_j s_{i,j} Z_{U^j}(w',w,y-x,\tau) + \sum_j s_{i,j} \hat{Z}_j(w',w,y-x,\tau) \right)
\end{aligned} \]
where $Z_U$ denotes $\sum_{i \in D} x_i Z_{U^i}$, $N = \wt(w) + \wt(w')$, $\kappa_i = e^{\wt(W \boxtimes U^i) - \wt(U^i)}$ for all $i \in D$, and $\hat{Z}_k(w,w',x-y,\tau)$ denotes the pseudo-trace function $\widehat{\Tr}_k(\cU(q_y) \cY_W(w,x-y)w',q_y) q^{L(0)-c/24}$.
\end{lem}
\begin{proof}
We obtain the following chain of equalities
\[ \begin{aligned}
\sum_{i \in D} &x_i \sum_j s_{i,j} Z_{W \boxtimes U^j}(\cY_{U^j}^2, w, w', x-y, \tau) \\
&= \sum_{i \in D} x_i \sum_j s_{i,j} Z_{U^j}(w',w,y-x-\tau,\tau) \\
&= \tau^{-N} Z_U(w',w,\frac{y-x-\tau}{\tau}, -1/\tau) \\
&= \tau^{-N} Z_U(w',w,\frac{y-x}{\tau}-1, -1/\tau) \\
&= \tau^{-N} \sum_{i \in D} x_i \kappa_i Z_{U^i}(w',w,\frac{y-x}{\tau}, -1/\tau) \\
&= \sum_{i \in D} x_i \kappa_i \left( \sum_j s_{i,j} Z_{U^j}(w',w,y-x,\tau) + \sum_j s_{i,j} \hat{Z}_j(w',w,y-x,\tau) \right)
\end{aligned} \]
where the first equality comes from Lemma \ref{lem:translate-by-tau}, the second equality is property NPT combined with Remark \ref{rem:modular}, the fourth equality is Lemma \ref{lem:translate-by-one} combined with property SW, and the fifth is the combination of Remark \ref{rem:modular} and Lemma \ref{lem:skew-symmetry}.
\end{proof}

\begin{lem} \label{lem:get-traces}
If $U= \sum_{i \in D} x_i U^i$ satisfies NPT and SW with respect to $W$, then the following two statements hold:
\begin{enumerate}
\item $\sum_{i \in D} x_i \sum_j s_{i,j} Z_{W \boxtimes U^j}(\cY_U^2, w, w', y-x, \tau)$ is a power series in $y-x$ whose coefficients are trace functions of $V$ on $V$-modules.  That is, there are no pseudo-traces, and the coefficients factor through $\pi_W: W \boxtimes W' \to V$.
\item The virtual module $\sum_{i \in D} x_i \kappa_i U^i$ satisfies NPT.
\end{enumerate}
\end{lem}
\begin{proof}
The sum $\sum_{i \in D} x_i \sum_j s_{i,j} Z_{W \boxtimes U^j}(\cY_{U^j}^2, w, w', y-x, \tau)$ that appears on the left side of Lemma \ref{lem:chain-of-equalities} is a linear combination of trace functions of $W \boxtimes W'$.  By this, we mean that the coefficient attached to $(y-x)^{-m-1}$ is
\[ \sum_{i \in D} x_i \sum_j s_{i,j}\Tr_{W \boxtimes U^j} \cY_{U^j}^2(\cU(q_y) w_m w', q_y) q^{L(0)-c/24}.\]
However, by Remark \ref{rem:modular}, the sum on the right side of that Lemma is a linear combination of trace and pseudo-trace functions of operators from $V$, where we mean that the coefficients of powers of $y-x$ are linear maps from $W \boxtimes W'$ to a space of holomorphic functions in $\tau$ that vanish on the kernel of $\pi_W: W \boxtimes W' \to V$.  Since the two sides are equal, we conclude that the contributions from pseudo-traces vanish, and all trace contributions come from $V$.  This proves the first claim.

For the second claim, we note that the previous paragraph yields the equality
\[ \tau^{-N} \sum_{i \in D} x_i \kappa_i Z_{U^i}(w',w,\frac{y-x}{\tau}, -1/\tau) = \sum_{i \in D} x_i \kappa_i \sum_j s_{i,j} Z_{U^j}(w',w,y-x,\tau) \]
If we choose our integral weight module $W$ to be $V$, then we obtain a decomposition of $Z_U(v,-1/\tau)$ into traces by setting $w' = \unit$.
\end{proof}

\begin{prop} \label{prop:tj-satisfies-npt}
If $V$ has non-negative $L(0)$-spectrum, then for every $j$, $T^{(j)}$ satisfies NPT.
\end{prop}
\begin{proof}
We note that each $T^{(j)}$ satisfies property SW with respect to any irreducible $V$-module $W$ because it is a simple current.  We would like to choose $W$ such that the following weight condition is satisfied:
\[ \wt(W \boxtimes T^{(j)}) - \wt(W) \equiv \frac{j}{n} \pmod{\bZ}. \]
If $W$ is an irreducible $V$-submodule of a $\sigma$-twisted $T$-module, then this is satisfied by the monodromy property of twisted modules.  By Theorem 1.1(ii) of \cite{DLM}, there exists a nonzero $\sigma$-twisted $T$-module, so there is some irreducible $V$-module $W$ satisfying the weight condition.

We claim that we may assume that $W$ has integral weight.  To show this we note that the lowest weight of $W$ is rational because $V$ is $C_2$-cofinite, so there is a lattice VOA $V_L$ (for $L$ positive-definite and even) and a simple $V_L$-module $V_{L + \alpha}$ such that $\wt(W) + \wt(V_{L+\alpha}) \in \bZ$.  Then $T \otimes V_L$ is a regular VOA with nonsingular invariant bilinear form, and $\hat{\sigma} = \sigma \otimes id_{V_L}$ is a finite order automorphism whose fixed-point subVOA is $V \otimes V_L$.  If $T^{(j)} \otimes V_L$ satisfies NPT, then $T^{(j)}$ satisfies NPT, so if $W$ has non-integral weight, we may replace $V$ with $V \otimes V_L$ and $W$ with $W \otimes V_{L + \alpha}$ for suitable $L$ and $\alpha$ in the following paragraph.

Set $\xi = e^{2\pi i/n}$.  We will prove that the virtual module $\sum_{j=0}^{n-1} \xi^{kj} T^{(j)}$ satisfies NPT for all $k$ by induction on $k$.  The case $k=0$ follows from the assumption that $T$ is regular.  Now, suppose $\sum_{j=0}^{n-1} \xi^{kj} T^{(j)}$ satisfies NPT for some $k$.  Thus, we obtain the equalities given in Lemma \ref{lem:chain-of-equalities}.   By the second claim of Lemma \ref{lem:get-traces}, we find that $\sum_{j=0}^{n-1} \xi^{(k+1)j} T^{(j)}$ also satisfies NPT, so the induction works.  By a discrete Fourier transform, we conclude that any virtual module given by a linear combination of $T^{(j)}$ satisfies NPT.
\end{proof}

\subsection{Rigidity}

Let $T$ be a non-negatively graded simple regular VOA with a finite automorphism $\sigma$ satisfying $T'\cong T$ and let $V$ be the fixed-point subVOA $T^\sigma$.  As we showed in Theorem \ref{thm:OPR2}, $V$ is projective, and $V' \cong V$ by the invariant form assumption.  Thus, for each simple $V$-module $W$, there is a unique pair given by a projection $\pi_W: W \boxtimes W' \to V$ and an embedding $e_W: V \to W \boxtimes W'$, such that $\pi_W \circ e_W = id_V$.  

Recall that for simple $V$-modules $A$, $B$, and $C$, there is an associator isomorphism 
\begin{equation} \label{eq:associator}
\mu_{A,B,C}:(A\boxtimes B)\boxtimes C \to A\boxtimes (B\boxtimes C)
\end{equation}
satisfying 
\[ \langle d',\mu_{A,B,C}\cY_{A\boxtimes B,C}^{\boxtimes}(\cY_{A,B}^{\boxtimes}(a,x_1)b,x_2)c\rangle
=\langle d', \cY_{A,B\boxtimes C}^{\boxtimes}(a,x_1-x_2)\cY_{B,C}^{\boxtimes}(b,x_2)c\rangle \]
in the open domain 
\[ \{(x_1,x_2)\in \bC^2 \mid 0<|x_1|<|x_2|<|x_1-x_2| \}\]
for any $a\in A, b\in B, c\in C$ and $d'\in (A\boxtimes (B\boxtimes C))'$, where $\langle\cdot,\cdot \rangle$ is the pairing on $(A\boxtimes(B\boxtimes C))'\times A\boxtimes(B\boxtimes C)$ (see Theorem 10.3 of \cite{HLZ6}).  A simple $V$-module $W$ is called ``rigid'' if the following composite is an isomorphism:

\[ W\boxtimes V \overset{{\rm id}_W \sbxt e_{W'}}{\longrightarrow} W 
\boxtimes (W'\boxtimes W) \overset{\mu}{\to} 
(W \boxtimes W') \boxtimes W \overset{\pi_W \sbxt {\rm id}_{W}}{\longrightarrow} V \boxtimes W \]

Our goal now is to prove that every simple $V$-module $W$ is rigid.  We first define some notation that we will use later. 

\begin{defn}
Set $D=e_W(V)\boxtimes W\subseteq (W\boxtimes W')\boxtimes W$. Then $D$ is a direct factor 
of $(W\boxtimes W')\boxtimes W$, and we let $i_D$ and $\pi_D$ denote the corresponding inclusion and projection maps.  Similarly, we set $F=\mu^{-1}(D)$ and let $i_F$ and $\pi_F$ denote the inclusion and projection maps.  We end up with the following diagram:
\[ \xymatrix{ F \ar[r]^{\mu|_F} \ar@<-1ex>[d]_{i_F} & D \ar@<-1ex>[d]_{i_D} \\ W\boxtimes (W'\boxtimes W) \ar@<-1ex>[u]_{\pi_F} \ar[r]^{\mu} & (W\boxtimes W')\boxtimes W \ar@<-1ex>[u]_{\pi_D} \ar[r]_-{\pi_W \boxtimes id_W} & V \boxtimes W \ar[ul]_{e_W \boxtimes id_W}
} \]
\end{defn}

\begin{lem} \label{lem:rigidity-with-piF}
$W$ is rigid if and only if $\Ker\pi_F$ does not contain $W\boxtimes e_{W'}(V)$.  
\end{lem}
\begin{proof}
This follows from the fact that $\mu|_F: F \to D$ and $e_W \boxtimes id_W: V \boxtimes W \to D$ are isomorphisms.
\end{proof}

Let us explain the rigidity condition from the viewpoint of intertwining operators. 
Recall that we introduced $\cY_W^1\in I\binom{W}{W, W'\boxtimes W}$ for $U=W$ in the proof of Lemma \ref{lem:translate-by-tau} as an intertwining operator satisfying 
\[ \langle u', Y^W \left(\cY_W(w,x-y)w',y\right)u\rangle
=\langle u', \cY_W^1(w,x)\cY_{W',W}^{\boxtimes}(w',y)u\rangle \]
for $u,w\in W, u',w'\in W'$. 

\begin{lem} \label{lem:YW1-as-fusion}
Let $\phi: W \to D$ be the composite of the canonical isomorphisms $W \to V \boxtimes W$ and $e_W \sbxt id_W$.  Then we have an equality $\phi \circ \cY_W^1 = \mu|_F \circ \pi_F \circ \cY_{W,W'\boxtimes W}^{\boxtimes}$ of intertwining operators of type $\binom{D}{W, W' \boxtimes W}$.
\end{lem}
\begin{proof}
By functoriality of $P(1)$-tensor product, the restriction of $\cY_{W\boxtimes W',W}^{\boxtimes}$ to the summand $e_W(V) \otimes W \subset (W \boxtimes W') \otimes W$ is given by $i_D \circ \cY_{e_W(V),W}^{\boxtimes}$.  Again by functoriality, $\phi \circ Y^W = \cY_{e_W(V),W}^{\boxtimes} \circ (e_W \otimes id_W)$ as intertwining operators of type $\binom{D}{V,W}$.  We therefore have $\phi \circ Y^W = \pi_D \circ \cY_{W\boxtimes W',W}^{\boxtimes} \circ (e_W \otimes id_W)$ as intertwining operators of type $\binom{D}{V,W}$.  By composing with $\cY_W$ we find that
\[ \pi_D \circ \cY_{W\boxtimes W',W}^{\boxtimes}(\cY_{W,W'}^{\boxtimes}(w,x-y)w',y)u = \phi \circ Y^W \left(\cY_W(w,x-y)w',y\right)u \]
as formal power series in $D$, and for any $d' \in D'$, we obtain an equality of holomorphic functions on $\cO_2$:
\[ \langle d', \pi_D \circ \cY_{W\boxtimes W',W}^{\boxtimes}(\cY_{W,W'}^{\boxtimes}(w,x-y)w',y)u\rangle = \langle d', \phi \circ Y^W \left(\cY_W(w,x-y)w',y\right)u \rangle. \]
By the defining property of $\cY_W^1$, we have the equality
\[ \langle d', \phi \circ Y^W \left(\cY_W(w,x-y)w',y\right)u\rangle
= \langle d', \phi \circ \cY_W^1(w,x)\cY_{W',W}^{\boxtimes}(w',y)u\rangle. \]
By associativity of fusion products, we have the equality
\[ \langle x', \cY_{W\boxtimes W',W}^{\boxtimes}(\cY_{W,W'}^{\boxtimes}(w,x-y)w',y)u \rangle = \langle x', \mu \circ \cY_{W,W'\boxtimes W}^{\boxtimes}(w,x)\cY_{W',W}^{\boxtimes}(w',y)u \rangle \]
for all $x' \in ((W \boxtimes W') \boxtimes W)'$, so setting $x' = d' \circ \pi_D$ yields
\[ \langle d', \phi \circ \cY_W^1(w,x)\cY_{W',W}^{\boxtimes}(w',y)u \rangle = \langle d', \mu|_F \circ \pi_F \circ \cY_{W,W'\boxtimes W}^{\boxtimes}(w,x)\cY_{W',W}^{\boxtimes}(w',y)u \rangle \]
By surjectivity of $\cY_{W',W}^{\boxtimes}$, the conclusion follows.
\end{proof}

\begin{lem}
The subspace $\Ker\cY_W^1:=\{u\in W'\boxtimes W \mid \cY^1_W(w,x)u=0 \, \forall w\in W\}$ 
is a $V$-submodule of $W'\boxtimes W$. 
\end{lem}

\begin{proof}
Since $\cY^1_W$ is an intertwining operator, 
for $p\in \Ker\cY^1_W, v\in V, m\in \bZ, w\in W$, we have 
\[ \cY^1_W(w,x) v_m p=v_m \cY^1_W(w,x)p - \sum_{j\in \bN}\binom{m}{j}\cY^1_W(v_jw,x)px^{-m+j}=0. \]
\end{proof}

The following is a relation between intertwining operators and rigidity. 

\begin{lem} \label{lem:rigidity-via-YW1}
$W$ is rigid if and only if $\Ker\cY_W^1$ does not contain $e_{W'}(V)$. 
\end{lem}

\begin{proof}
If $\Ker\pi_F$ contains $W\boxtimes e_{W'}(V)$, then $\Ker \pi_F \cY_{W,W'\boxtimes W}^{\boxtimes}$ contains $e_{W'}(V)$, because all of the coefficients vanish.  If $\Ker \pi_F \cY_{W,W'\boxtimes W}^{\boxtimes}$ contains $e_{W'}(V)$, then by surjectivity of the restriction $\cY_{W,e_{W'}(V)}^{\boxtimes}$ of $\cY_{W,W'\boxtimes W}^{\boxtimes}$ to $W \otimes e_{W'}(V)$, $\Ker\pi_F$ must contain $W\boxtimes e_{W'}(V)$.  Lemma \ref{lem:YW1-as-fusion} yields an equality between $\Ker \cY_W^1$ and $\Ker \pi_F \cY_{W,W'\boxtimes W}^{\boxtimes}$, so the result follows from the rigidity criterion of Lemma \ref{lem:rigidity-with-piF}.
\end{proof}

We now prove the following main result of this subsection.  Let $s_{TT}$ denote the S-matrix coefficient of $Z_T(v,\tau)$ in the expansion of $\tau^{\wt[v]} Z_T(v,-1/\tau)$ into a linear combination of trace functions on $T$-modules.  

\begin{prop} \label{prop:rigidity}
If $V$ has non-negative $L(0)$-spectrum and $s_{TT}$ is nonzero, then every simple $V$-module is rigid.  In particular, if $T$ is regular and has non-negative $L(0)$-spectrum, then every simple $V$-module is rigid.
\end{prop}

\begin{proof}
Let $W$ be a simple $V$-module.  We may assume that $W$ has integral weights, by following the same method as in the proof of Proposition \ref{prop:tj-satisfies-npt}.  That is, by taking tensor products with a suitable lattice VOA $V_L$ and a rigid simple module $V_{L+\alpha}$, we obtain a simple $V \otimes V_L$-module $W \otimes V_{L + \alpha}$ with integral weights, and $W$ is rigid as a $V$-module if and only if $W \otimes V_{L + \alpha}$ is rigid as a $V \otimes V_L$-module.

By Proposition \ref{prop:tj-satisfies-npt}, every $T^{(i)}$ satisfies condition NPT and so there is $\lambda_U\in \bC$ for each simple $V$-module $U$ such that 
\[(\tau)^{-N}Z_{T^{(i)}}(w,w',\frac{x-y}{\tau},-1/\tau)=\sum_{U\in {\rm irr}(V)} \lambda_U Z_U(w,w',x-y,\tau), \]
where ${\rm irr}(V)$ denotes the set of simple $V$-modules and $N=\wt[w]+\wt[w']$. 

By our hypothesis that $s_{TT} \neq 0$, the $S$-transformation of
\[ Z_T(w,w',x-y,\tau)=\sum_{j=0}^{n-1} Z_{T^{(j)}}(w,w',x-y,\tau) \] 
is a linear combination of trace functions on $T$-modules and $Z_T(w,w',x-y,\tau)$ appears with a nonzero 
coefficient. In particular, by decomposing $T$-modules as $V$-modules, $Z_V(w,w',x-y,\tau)$ appears with a nonzero coefficient.  Therefore, there is some $T^{(j)}$ such that  $Z_V(w,w',x-y,\tau)$ appears with nonzero coefficient in the 
$S$-transformation $Z_{T^{(j)}}(w,w',\frac{x-y}{\tau},-1/\tau)$. We fix $j$ and write 
\[(\tau)^{-N}Z_{T^{(j)}}(w,w',\frac{x-y}{\tau},-1/\tau)=\sum_{U\in {\rm irr}(V)} \lambda_U Z_U(w,w',x-y,\tau) \]
where $\lambda_V \neq 0$. Since $T^{(j)}$ is a simple current, Lemma \ref{lem:translate-by-one} implies there is some $\xi\in \bC \setminus \{0\}$ such that 
\[Z_{T^{(j)}}(w,w',x-y+1,\tau)=\xi Z_{T^{(j)}}(w,w',x-y,\tau)\]

We consider the $S$-transformation of $Z_{T^{(j)}}(w,w',x-y,\tau)$ in two different ways.  First, we translate by $1$, apply $S$, apply skew-symmetry and translate by $\tau$:
\[ \begin{aligned}
(\tau)^{-N}\xi &Z_{T^{(j)}}(w,w',\frac{x}{\tau}-\frac{y}{\tau},-1/\tau) \\
&= (\tau)^{-N}Z_{T^{(j)}}(w,w',\frac{x}{\tau}-\frac{y}{\tau}+1,-1/\tau) \\
&= (\tau)^{-N}Z_{T^{(j)}}(w,w',\frac{x-y+\tau}{\tau},-1/\tau) \\
&= \sum_U \lambda_U Z_U(w,w',x-y+\tau, \tau) \\
&= \sum_U \lambda_U Z_U(w',w,y-x-\tau, \tau) \\
&= \sum_U \Tr_{W\boxtimes U}\cY^2_U(\cU(q_y)\cY_{W,W'}^{\boxtimes}(w,x-y)w',q_y)q^{L(0)-c/24}. 
\end{aligned} \]
Then we just apply $S$:
\[ \begin{aligned}
(\tau)^{-N}\xi &Z_{T^{(j)}}(w,w',\frac{x}{\tau}-\frac{y}{\tau},-1/\tau) \\
&= \xi\sum_U \lambda_U Z_U(w,w',x-y,\tau) \\
&= \xi\sum_U \lambda_U \Tr_U Y(\cU(q_y)\cY_W (w,x-y)w',q_y)q^{L(0)-c/24}
\end{aligned} \]
Therefore, we have 
\[\begin{aligned}
\xi \sum_{U\in {\rm irr}(V)} & \lambda_U \Tr_U Y(\cU(q_y)\cY_W(w,x-y)w',q_y)q^{L(0)-c/24} \\
&=\sum_{U\in {\rm irr}(V)} 
\Tr_{W\boxtimes U}\cY^2_U(\cU(q_y)\cY_{W,W'}^{\boxtimes}(w,x-y)w',q_y)q^{L(0)-c/24}.
\end{aligned}\]
By comparing coefficients of the expansions in $x-y$ and $\log(x-y)$, we see that for any $r\in W\boxtimes W'$, we have 
\[ \xi \sum_{U\in {\rm irr}(V)} \lambda_U \Tr_U Y(\pi_W(r),z)q^{L(0)-c/24}=\sum_{U\in {\rm irr}(V)} \Tr_{W\boxtimes U}\cY^2(r,z)q^{L(0)-c/24}. \]
In particular, we may restrict to $r\in e_W(V)$ to find that
\[ \begin{aligned}
\xi \sum_{U\in {\rm irr}(V)} &\lambda_U \Tr_U Y(\pi_W(r),z)q^{L(0)-c/24} \\
&= \sum_{U\in {\rm irr}(V)} \Tr_{W\boxtimes U}\cY^2(e_W\pi_W(r),z)q^{L(0)-c/24}
\end{aligned} \]
and so
\begin{equation} \label{eq:Moore-Seiberg-comparison}
\begin{aligned}
\xi \sum_{U\in {\rm irr}(V)} &\lambda_U \Tr_U Y(\cU(q_y)\cY_W(w,x-y)w',q_y)q^{L(0)-c/24} \\
&= \sum_{U\in {\rm irr}(V)} \Tr_{W \boxtimes U}\cY^2_U(e_W \pi_W \cU(q_y) \cY_{W,W'}^{\boxtimes}(w,x-y)w',q_y)q^{L(0)-c/24}
\end{aligned}
\end{equation}

The composite of $\cY^2_U$ with $e_W$ applied to the first input is an intertwining operator of the form $I\binom{W\boxtimes U}{V, W\boxtimes U}$.
Since the trace only depends on the semisimplification of a module, we can decompose our trace into a sum of traces over simple composition factors of $W \boxtimes U$.  That is, we replace this composite with a linear combination $\sum a_P Y^P$ of structure maps $Y^P$ for the action of $V$ on the simple composition factors $P$ of $W\boxtimes U$, yielding
\[ \begin{aligned}
\xi \sum_{U\in {\rm irr}(V)} &\lambda_U \Tr_U Y(\cU(q_y)\cY_W(w,x-y)w' q^{L(0)-c/24} \\
&= \sum_{U\in {\rm irr}(V)} \sum_{P \in \text{Comp}(W \boxtimes U)} a_P Z_P(w,w',x-y,\tau).
\end{aligned} \]

By our choice of $T^{(j)}$, we have $\lambda_V \neq 0$.  By Lemma \ref{lem:linear-independence}, the traces $Z_P (w,w',x-y,\tau)$ are linearly independent, so there is some $U$ such that there is a composition factor $P \cong V$ with nonzero coefficient $a_P$.  Since $W \boxtimes U$ then has a composition factor isomorphic to $V$, the projectivity of $V$ (from Theorem \ref{thm:OPR2}) implies $V$ is a submodule, and self-duality of $V$ then implies $V$ is a direct summand.  Thus, there is a nonzero intertwining operator of type $\binom{V}{W, U}$.  By the duality properties of intertwining operators, the irreducible module $U$ is then necessarily isomorphic to $W'$, so $P = e_W(V) \subseteq W \boxtimes W'$.  Thus, we have the following nonvanishing summand on the right side of Equation \eqref{eq:Moore-Seiberg-comparison}:
\[ \Tr_P\cY^2_{W'}\left(\cU(q_y)\pi_W\cY_{W,W'}^{\boxtimes}(w,x-y)w',q_y\right)q^{L(0)-c/24} \neq 0. \]
By the properties of $\cY^2$ given in the proof of Lemma \ref{lem:translate-by-tau}, the left side is equal to 
\[\Tr_P \cY_{W,W'}^{\boxtimes}(\cU(q_x)w,q_x)\cY_{W'}^1(\cU(q_y)w',q_y)q^{L(0)-c/24}\]
and so $\cY^1_{W'}(w',x)P \neq 0$.  This implies $\Ker \cY^1_{W'}$ does not contain $e_W(V)$, so by the criterion of Lemma \ref{lem:rigidity-via-YW1}, $W'$ is rigid.  Passage to contragradients is an involutive operation on isomorphism classes of irreducible $V$-modules, so this resolves the first claim.

The second claim then follows from Lemma \ref{lem:stt-nonzero}, i.e., if $T$ is regular and has non-negative $L(0)$-spectrum, then $s_{TT} \neq 0$.  This is where the Verlinde conjecture is used (i.e., Theorem 5.5 of \cite{HV}, but with slightly weaker hypotheses).
\end{proof}

\subsection{Main Theorem}

\begin{thm} \label{thm:main}
Let $T$ be a simple regular VOA with non-negative $L(0)$-spectrum and a nonsingular invariant form.  Let $\sigma$ be a finite order automorphism of $T$.  Then $V$ is regular.
\end{thm}

\begin{proof}
By Proposition \ref{prop:regularity-from-projectivity}, it suffices to show that any simple $V$-module $W$ is projective in the category of $V$-modules.

Consider a $V$-module surjection $f: D \to W$.  We wish to show that it splits.
By taking its pullback along
\[ \pi_W\sbxt 1_{W\boxtimes U}:(W\boxtimes W')\boxtimes W\to V\boxtimes W, \]
we have the following commutative diagram: 
\[ \xymatrix{ & W\boxtimes (W'\boxtimes D) \ar[r]^{\mu_D} \ar[d]_{1_W \sbxt (1_{W'} \sbxt f)}
& (W\boxtimes W')\boxtimes D \ar[r]^-{\pi_W\sbxt 1_{D}} \ar[d]_{1_{W\boxtimes W'}\sbxt f }
& V \boxtimes D \ar[d]_{1_V \sbxt f} \\
W \boxtimes V \ar[r]_-{1_W \sbxt e_W} & W\boxtimes (W'\boxtimes W) \ar[r]_{\mu}
& (W\boxtimes W') \boxtimes W \ar[r]_-{\pi_W \sbxt 1_W} & V\boxtimes W } \]
where $\mu_D$ is an isomorphism making the left square commute (and this exists by functoriality and associativity of fusion).  We will show that a splitting of $f$ exists by finding a submodule of $W \boxtimes (W' \boxtimes D)$ that is isomorphic to $W$, such that the diagram produces an isomorphism to $V \boxtimes W$ that factors through $V \boxtimes D$.

Consider the set $\fX$ of submodules $Q \subseteq W'\boxtimes D$ such that $1_{W}\sbxt (1_{W'} \sbxt f)(W \boxtimes Q) = W \boxtimes e_W(V)$.  This set is nonempty, since we can just take the pullback.  For any such submodule $Q \in \fX$, we have a surjection $1_{W'} \sbxt f: Q \to e_W(V)$, and since $V$ is projective by Theorem \ref{thm:OPR2}, there is a section $V \to Q$.  In particular, if we take $Q$ to be a minimal element of $\fX$, then $Q$ is isomorphic to $V$.

Under this choice of $Q$, the rigidity asserted in Proposition \ref{prop:rigidity} gives us an isomorphism
\[ (\pi_W\sbxt 1_W) \circ \mu\circ (1_W\sbxt(1_{W'}\sbxt f)): W\boxtimes Q \to V \boxtimes W. \]
By commutativity of the diagram, $(\pi_W\sbxt 1_D)\circ \mu_D(W\boxtimes Q)$ is a submodule of $V \boxtimes D$ that is isomorphic to $W$, and maps isomorphically to $V \boxtimes W$ along $1_V \sbxt f$.  Thus, $f:D\to W$ splits. 
\end{proof}

The cyclic orbifold problem for a property $P$ naturally extends to the solvable orbifold problem for $P$:

\begin{cor} \label{cor:solvable-fixed-points}
Let $T$ be a VOA that is simple, regular, with non-negative $L(0)$-spectrum, and self-dual as a $T$-module.  Let $G$ be a finite solvable group of automorphisms of $T$.  Then the fixed-point subVOA $T^G$ is simple, regular, with non-negative $L(0)$-spectrum, and self-dual as a $T^G$-module.
\end{cor}
\begin{proof}
We may reduce to the case that $G$ is cyclic of prime order by induction along the composition series of $G$.  We may then apply Theorem \ref{thm:main}.
\end{proof}

\begin{cor} \label{cor:g-rational}
Let $T$ be a simple regular VOA with non-negative $L(0)$-spectrum and a non-singular invariant form.  Then $T$ is $\sigma$-rational for any finite-order automorphism $\sigma$.  That is, all admissible $\sigma$-twisted $T$-modules are completely reducible.  Furthermore, $T$ is $\sigma$-regular, in the sense that all weak $\sigma$-twisted $T$-modules are direct sums of ordinary irreducible $\sigma$-twisted $T$-modules.
\end{cor}

The claim that Theorem \ref{thm:main} implies this corollary is proved (under slightly different hypotheses) as Lemma 4.2 in \cite{ADJR}.  Our argument follows essentially the same lines, and we suspect it was known to Dong, Li, and Mason in the 1990s.

\begin{proof}
By Theorem \ref{thm:main}, $T^\sigma$ is regular, hence rational, so by Theorem 4.10(a) of \cite{DLM2}, the associative algebra $A_k(T^\sigma)$ is finite dimensional and semisimple for all integers $k$.  As mentioned in the proof of Proposition 2.5 of \cite{DLM1}, the associative algebra $A_{\sigma,n}(T)$ is a quotient of $A_{\lfloor n \rfloor}(T^\sigma)$ for all $n \in \frac{1}{|\sigma|}\bZ_{\geq 0}$.  Thus, each $A_{\sigma,n}(T)$ is also semisimple.  By Theorem 4.5 of \cite{DLM1}, $\sigma$-rationality is equivalent to semisimplicity of the associative algebras $A_{\sigma,n}(T)$ for all $n \in \frac{1}{|\sigma|}\bZ_{\geq 0}$ - a proof is not given there, but the proof of Theorem 4.11 of \cite{DLM2}, which concerns the untwisted case, translates without significant change.  The $\sigma$-regularity claim follows from the fact that weak $\sigma$-twisted $T$-modules are admissible (by essentially the same argument as Proposition \ref{prop:regularity-from-projectivity}) and the fact that irreducible admissible $\sigma$-twisted $T$-modules are ordinary.
\end{proof}

\section{Application: holomorphic orbifolds and Generalized Moonshine}

We briefly describe an application of Corollary \ref{cor:g-rational}.  In \cite{N}, Norton formulated a ``Generalized Moonshine'' conjecture based on extensive computations relating centralizers of elements in the monster simple group $\bM$ with genus zero modular functions.

\begin{conj}
There is a rule that assigns to each element $g \in \bM$ a graded projective representation $V(g) = \bigoplus_{n \in \bQ} V(g)_n$ of the centralizer $C_{\bM}(g)$ and to each commuting pair $(g,h)$ in $\bM$ a holomorphic function $Z(g,h;\tau)$ on the complex upper half-plane $\fH$, such that the following 5 properties hold:
\begin{enumerate}
\item There is a lift $\tilde{h}$ of $h$ to a finite order linear transformation on $V^\natural(g)$ such that $Z(g,h;\tau) = \Tr_{V^\natural(g)}\tilde{h} q^{L_0-1}$ (where $q = e^{2\pi i \tau}$).
\item The function $(g,h) \mapsto Z(g,h;\tau)$ is invariant under simultaneous conjugation on the pair $(g,h)$ in $\bM$, up to rescaling.
\item $Z(g,h;\tau)$ is either a constant or a Hauptmodul (meaning $Z(g,h;\tau)$ is invariant under a discrete group $\Gamma_{g,h} \subset SL_2(\bZ)$ containing some $\Gamma(N)$, such that $Z(g,h;\tau)$ generates the function field of the quotient $\Gamma_{g,h}\backslash \fH$).
\item For any $\left(\begin{smallmatrix} a & b \\ c & d \end{smallmatrix} \right) \in SL_2(\bZ)$ and any commuting pair $(g,h)$ in $\bM$, $Z(g,h,\frac{a\tau+b}{c\tau+d})$ is proportional to $Z(g^a h^c, g^b h^d,\tau)$.
\item $Z(g,h;\tau)$ is proportional to $J(\tau)$ if and only if $g=h=1$.
\end{enumerate}
\end{conj}

In \cite{DGH88}, this conjecture was quickly given a physical interpretation: $V(g)$ is the Hilbert space of the $g$-twisted sector of a conformal field theory with $\bM$-symmetry, and $Z(g,h;\tau)$ is the genus one partition function with boundary conditions twisted by the commuting pair $(g,h)$.  At a physical level of rigor, claims 1, 2, 4 and 5 then follow from orbifold conformal field theory considerations.  In fact, these considerations were introduced in \cite{V86}, where the compatibility $Z(g,h,\frac{a\tau+b}{c\tau+d}) = Z(g^a h^{-c}, g^{-b} h^d,\tau)$ was discussed and the possibility of phase anomalies yielding a proportionality constant was introduced.  The difference in signs amounts to a different convention for twisting, where some authors define a $g$-twisted module by the same formula that others use to define $g^{-1}$-twisted module.  In a geometric sense, Norton's convention amounts to defining elliptic curves by lattice quotients $\bC/\Lambda$ with distinguished basis given by $\Lambda = \langle -1, \tau \rangle$ instead of the usual choice $\Lambda = \langle 1, \tau \rangle$.

On the mathematical side, this interpretation amounts to setting $V(g)$ to be an irreducible $g$-twisted $V^\natural$-module, where $V^\natural$ is the vertex operator algebra with $\bM$ symmetry constructed in \cite{FLM}.  The main advance toward proving this interpretation of the conjecture came in \cite{DLM}, where it was shown that for any $g \in \bM$, there exists an irreducible $g$-twisted $V^\natural$-module, and such modules are unique up to isomorphism.  This gives a natural projective $C_{\bM}(g)$ action on the twisted modules, and after some work on convergence, claims 1, 2, and 5 were resolved by defining $Z$ to satisfy claim 1.  Furthermore, results in the same paper reduced claim 4 to the question of $g$-rationality of $V^\natural$.  More generally, they showed that similar properties hold for $g$-rational holomorphic $C_2$-cofinite vertex operator algebras.  We now resolve claim 4 unconditionally.

\begin{thm}
$V^\natural$ is $g$-rational for all $g \in \bM$.
\end{thm}
\begin{proof}
By Corollary \ref{cor:g-rational}, it suffices to show that $V^\natural$ is non-negatively graded, holomorphic, and $C_2$-cofinite.  The non-negativity of the grading is immediate from the construction in \cite{FLM}, the holomorphic property is proved in \cite{D94}, and the $C_2$-cofinite property is proved as Proposition 12.4 of \cite{DLM}
\end{proof}

\begin{thm} \label{thm:sl2z-rule}
Let $V$ be a holomorphic $C_2$-cofinite VOA with non-negative $L(0)$-spectrum.  Given a commuting pair $(g,h)$ of finite order automorphisms of $V$, let $Z(g,h;\tau) = \Tr_{V(g)}\tilde{h} q^{L_0-1}$, where $V(g)$ is an irreducible $g$-twisted $V$-module, and $\tilde{h}$ is some lift of $h$ to a linear transformation on $V(g)$.  Then for any $\left(\begin{smallmatrix} a & b \\ c & d \end{smallmatrix} \right) \in SL_2(\bZ)$, $Z(g,h,\frac{a\tau+b}{c\tau+d})$ is proportional to $Z(g^a h^c, g^b h^d,\tau)$.
\end{thm}
\begin{proof}
Theorem 10.1 of \cite{DLM} asserts that if $V$ is $C_2$-cofinite and $g$-rational, and if $M^1, \ldots, M^m$ represent a complete list of isomorphism classes of $h$-stable irreducible $g$-twisted $V$-modules, then the corresponding trace functions $T_1,\ldots,T_m$ span the space $\cC_1(g,h)$ of $(g,h)$-twisted genus 1 functions.  Under our hypotheses, $V$ is automatically $C_2$-cofinite, and it is $g$-rational by Corollary \ref{cor:g-rational}.  By Theorem 10.3 of \textit{loc. cit.}, there is a unique isomorphism class of irreducible $g$-twisted $V$-module, so we find that $\cC_1(g,h)$ is spanned by the corresponding trace function $T_{g,h}(v;\tau) = \Tr(o(v)\tilde{h}q^{L(0)-c/24}|V(g))$.  By Theorem 5.4 of \textit{loc. cit.}, there is an action of $SL_2(\bZ)$ on the set of genus 1 function spaces, such that $(c\tau+d)^{\wt v}T_{g,h}(v;\frac{a\tau+b}{c\tau+d}) \in \cC_1(g^a h^c,g^bh^d)$ for all $\left(\begin{smallmatrix} a & b \\ c & d \end{smallmatrix} \right) \in SL_2(\bZ)$.  The conclusion then follows by noting that $Z(g,h;\tau) = T_{g,h}(\unit;\tau)$.
\end{proof}

\begin{cor} \label{cor:generalized-moonshine}
Claim 4 in the Generalized Moonshine conjecture is true, when we set $V(g)$ to be an irreducible $g$-twisted $V^\natural$-module.
\end{cor}

\begin{rem}
Much more can be said about the proportionality constants that appear in Claims 2 and 4 in the Generalized Moonshine conjecture.  In particular, \cite{GPV13} conjectured that the constants that would appear in the corresponding claims for any holomorphic $C_2$-cofinite vertex operator algebra are controlled by a distinguished element in $H^3(\Aut(V),\bC^\times)$.  This is now firmly established in \cite{vEMS} for the case of a finite cyclic group of automorphisms.
\end{rem}

\end{document}